\documentclass[11pt,reqno]{amsart}
\usepackage[utf8]{inputenc}

\usepackage[left=3cm, right=3cm]{geometry}

\usepackage{amsfonts, amsthm, amsmath, amssymb,bm}
\usepackage{amssymb,amscd}
\usepackage{mathtools}
\usepackage[numbers]{natbib}
\usepackage{esint}
\usepackage{mathabx}

\usepackage{xcolor}

\usepackage{enumerate}
\usepackage{cases}

\newcommand{\Aa}{\mathbb{A}}

\newcommand{\Zz}{\mathbb{Z}}
\newcommand{\Rr}{\mathbb{R}}
\newcommand{\Cc}{\mathbb{C}}
\newcommand{\Qq}{\mathbb{Q}}

\newcommand{\Ima}{\operatorname{Im}}
\newcommand{\Ree}{\operatorname{Re}}

\newcommand{\bs}{\backslash}
\newcommand{\mc}{\mathcal}
\newcommand{\mf}{\mathfrak}

\renewcommand{\d}{\mathrm{d}}
\newcommand{\GL}{\mathrm{GL}}
\newcommand{\PGL}{\mathrm{PGL}}
\newcommand{\G}{\mathrm{G}}
\newcommand{\N}{\mathrm{N}}

\newtheorem{thm}{Theorem}[section]

\newtheorem{prop}[thm]{Proposition}
\newtheorem{cor}[thm]{Corollary}
\newtheorem{rmk}{Remark}[section]

\newtheorem{hyp}{Hypothesis}

\newcommand{\try}[1]{{\color{purple} [#1]}}

\newcommand\mtrx[1]{\begin{pmatrix} #1 \end{pmatrix}} 
\newcommand\IP[1]{\langle #1\rangle}

\begin{document}
\title{Spectral reciprocity via integral representations}
\author{Ramon M. Nunes} \thanks{This work was partially supported by the DFG-SNF lead agency program grant 200021L-153647.}
\address{
   Ramon M. Nunes // Departamento de Matemática, Universidade Federal do Ceará (UFC). Bloco 914, Campus do Pici // Bloco 914 // CEP 60455-760 Fortaleza - CE, Brasil}
  \email{ramon@mat.ufc.br}

\date{\today}

\begin{abstract}
We prove a spectral reciprocity formula for automorphic forms on $\GL(2)$ over a number field that is remininscent of the one found by Blomer and Khan. Our approach uses period representations of $L$-functions and the language of automorphic representations.
\end{abstract}

\maketitle

\section{Introduction}

In the past few years, some attention has been given to spectral reciprocity formulae. By this we mean an identity of the shape
\begin{equation}\label{SpecRec}
 \sum_{\pi \in \mathcal{F}}\mathcal{L}(\pi)\mathcal{H}(\pi) = \sum_{\pi \in \widetilde{\mathcal{F}}}\widetilde{\mathcal{L}}(\pi)\widetilde{\mathcal{H}}(\pi), 
\end{equation}  
where $\mathcal{F}$ and $\widetilde{\mathcal{F}}$ are families of automorphic representations, $\mathcal{L}(\pi)$ and  $\widetilde{\mathcal{L}}(\pi)$ are certain $L$-values associated to $\pi$, $\mathcal{H}$ and $\widetilde{\mathcal{H}}$ are some weight functions.

The term \emph{spectral reciprocity} first appeared in this context in a paper by Blomer, Li and Miller \cite{BML2019spectral} but such identities have been around at least since Motohashi's formula connecting the fourth moment of the Riemann zeta-function to the cubic moment of L-functions of cusp forms for ${\rm GL}(2)$ ({\em cf.} \cite{Mot1993explicit}).

The more recent results concern the cases where the families $\mathcal{F}$ and $\widetilde{\mathcal{F}}$ are the same or nearly the same. Most commonly, these families are taken to be formed by automorphic representations of $\GL(2)$.

There are at least two reasons that help understand the appeal of such formulae. The first one is that they give a somewhat conceptual way of summarizing a technique often used in dealing with problems on families of $\GL(2)$ $L$-functions in which one uses the Kuznetsov formula on both directions in order to estimate a moment of $L$-values. The second one comes from their satisfying intrinsic nature relating objects that have no \textit{a priori} reason to be linked.

The first versions of these $\GL(2)$ - spectral reciprocity formulae (\citep{BlKh2019reciprocity}, \citep{AK2018level} and \citep{BlKh2019uniform}) used classical techniques such as the Voronoi summation formula and the Kuznetsov formula. Starting from \citep{Zach2019Periods}, it became clear that an adelic approach could be of interest. Not only this renders generalization to number fields almost immediate, it can also avoid some of the combinatorial difficulties that arise when applying the Voronoi formula.

In \cite{BlKh2019reciprocity}, Blomer and Khan have shown a reciprocity formula which is the main inspiration for the present work: Let $\Pi$ be a fixed automorphic representation of ${\rm GL}(3)$ over $\mathbb{Q}$. Let $q$ and $\ell$ be coprime integers. We write

$$
\mathcal{M}(q,\ell;h):=\frac{1}{q}\sum_{\text{cond}(\pi)=q}\frac{L(1/2,\Pi\times\pi)L(1/2,\pi)}{L(1,\text{Ad},\pi)}\frac{\lambda_{\pi}(\ell)}{\ell^{1/2}}h(t_{\pi})+(\cdots),
$$
where
\begin{itemize}
  \item $\pi$ runs over cuspidal automorphic representations of $\PGL(2)$,
  \item $\lambda_{\pi}(\ell)$ is the eigenvalue of the Hecke-operator $T_{\ell}$ on $\pi$,  
  \item $t_\pi$ is the spectral parameter,
  \item $h$ is a {\em fairly general} smooth function and
  \item $(\cdots)$ denotes the contribution of the Eisenstein part, the terms of lower conductor and some degenerate terms.
\end{itemize}

Then, Blomer and Khan have showed that

$$
\mathcal{M}(q,\ell,h)=\mathcal{M}(\ell,q,\check{h}),
$$
where $h\mapsto \check{h}$ is given by an explicit integral transformation. When $\Pi$ corresponds to an Eisenstein series, this has an application to subconvexity: Let $\pi$ be a cuspidal automorphic respresentation for ${\rm GL}(2)$ over $\mathbb{Q}$ of {\em squarefree} conductor, then:
\begin{equation}\label{subconvexity}
L(1/2,\pi)\ll_{\epsilon}(\text{cond} (\pi))^{\frac14-\frac{1-2\vartheta}{24}+\epsilon},
\end{equation}
where $\vartheta$ is an admissible exponent towards the Ramanujan conjecture (we know that $\frac{7}{64}$ is admissible and $\vartheta=0$ corresponds to the conjecture). This was then the best know bound of its kind but it was later superseeded by the one in \cite{BHKM2020MOtohashi}.

In this article we use the theory of adelic automorphic representations and integral representations of Rankin-Selberg $L$-functions to deduce a result on number fields of similar flavor to that of \cite[Theorem 1]{BlKh2019reciprocity}. 

With respect to Blomer and Khan's result, our result has the advantage of being valid for any number field. On the other hand we need to make some technical restrictions that prevent us from having a full generalization of their reciprocity formula. For the moment our results only work when the fixed $\GL(3)$ form is cuspidal and our formula only contemplates forms that are spherical at every infinity place. The first restriction is made for analytic reasons and is due to the fact that unlike cusp forms, the Eisensten series are not of rapid decay. This can probably be resolved by means of a suitable notion of regularized integrals. As for the second restriction, this seems to be of a more representation-theoretic nature. It requires showing analyticity of certain local factors for non-unitary representations of $\GL(2)$. We hope to address both of these technical issues in future work.

\subsection{Statement of results}

Let $F$ be a number field, with ring of integers $\mf{o}_F$. Let $\Pi$ be a cuspidal automorphic representation of $\GL(3)$ over $F$.
%
%We will eventually impose that $\Phi$ is also unramified at the places at $\infty$ but as we will see this has the effeect of erasing the contribution of the representations of $GL(2)$ that are not unramified at the places at $\infty$.
%
For each automorphic representation $\pi$ of $\GL(2)$, we consider the completed $L-$functions
$$
\Lambda(s,\pi),\;\Lambda(s,\operatorname{Ad},\pi)\text{ and }\Lambda(s,\Pi\times\pi).
$$
These are, respectively, the Hecke $L-$function and the Adjoint $L-$function of $\pi$, and the Rankin-Selberg $L-$function of $\Pi \times \pi$, where for the Rankin-Selberg $L$-functions we take the naive definition \eqref{L-naive}. These coincide with the local $L$-functions à la Langlands at all the unramified places but might differ at the ramified ones. Notice that this might also affect the values of $L(s,\mathrm{Ad},\pi)$. 

Let $\xi_F$ denote the completed Dedekind zeta function of $F$ and let $\xi_F^{\ast}(1)$ denote its residue at $1$. Let $\Phi\simeq\otimes_v\Phi_v$ be a vector in the representation space of $\Pi$. Let $s$ and $w$ be complex numbers and let $H$ denote the weight function given by \eqref{H-def}. We consider the following sums 

$$
\mc{C}_{s,w}(\Phi):=\sum_{\pi\in C(S)}\frac{\Lambda(s,\Pi\times \pi)\Lambda(w,\pi)}{\Lambda(1,\operatorname{Ad},\pi)}H(\pi)
$$
and
\begin{equation}\label{E(H)-def}
  \mc{E}_{s,w}(\Phi):=\sum_{\omega\in \Xi(S)}\int_{-\infty}^{\infty}\frac{\Lambda(s,\Pi\times \pi(\omega, it))\Lambda(w,\pi(\omega,it))}{\Lambda^{\ast}(1,\operatorname{Ad},\pi(\omega,it))}H(\pi(\omega,it))\frac{\d t}{2\pi},
\end{equation}
where $S$ is any finite set of places containing all the archimedean ones and those for which $\Phi_v$ is ramified, $C(S)$ (resp. $\Xi(S)$) denotes the collection of cuspidal automorphic representations of $\GL(2)$ (resp. unitary normalized idele characters) over $F$ that are unramified everywhere outside $S$. Finally, $\pi(\omega,it)$ denotes a normalized induced representation as in \S \ref{Ind-Eis} and $\Lambda^{\ast}(1,\operatorname{Ad},\pi)$ denotes the first non-zero Laurent coefficient of $\Lambda(s,\operatorname{Ad},\pi)$ at $s=1$. The main object of study in this work is the following ``moment'':
\begin{equation}\label{M=C+E}
  \mc{M}_{s,w}(\Phi):=\mc{C}_{s,w}(\Phi)+\mc{E}_{s,w}(\Phi).
\end{equation}

We remark that the values of $\Lambda(s,\Pi\times \pi(\omega,it))$, $\Lambda(w,\pi(\omega,it))$ and  $\Lambda^{\ast}(1,\operatorname{Ad}, \pi(\omega,it))$ can be given in terms of simpler $L-$functions as follows:

\begin{align*}
\Lambda(s,\Pi\times \pi(\omega,it))&=\Lambda(s+it,\Pi\times \omega)\Lambda(s-it,\Pi\times \overline{\omega})\\
\Lambda(w,\pi(\omega,it))&=\Lambda(w+it,\omega)\Lambda(w-it,\overline{\omega}),\\
\Lambda^{\ast}(1,\operatorname{Ad},\pi(\omega,it))&=\operatorname{Res}_{s=1}\left[\Lambda(s+2it,\omega^2)\Lambda(s-2it,\overline{\omega}^2) \xi_F(s)\right]\\
                                                  &=\Lambda(1+2it,\omega^2)\Lambda(1-2it,\overline{\omega}^2) \xi^{\ast}_F(1),\,\,(t\neq 0).
\end{align*}
where $\Lambda(s,\Pi\times \omega)$ and $\Lambda(s,\omega)$ are the (completed) Rankin-Selberg $L$-function of $\Pi\times \omega$ and Dirichlet $L$-function of $\omega$, respectively.

We start with the following result which can be seen as a preliminary reciprocity formula.

\begin{thm}\label{main}
Let $s,w\in\Cc$ and define

\begin{equation}\label{s'w'}
(s',w'):=\left(\frac{1+w-s}{2},\frac{3s+w-1}{2}\right).
\end{equation}
Let $H$ be as in \eqref{H-def} and $\widecheck{H}$ be given by \eqref{H-check}. Suppose the real parts of $s$, $w$, $s'$ and $w'$ are sufficiently large. Then we have the relation

$$
\mc{M}_{s,w}(\Phi)+\mathcal{D}_{s,w}(\Phi)=\mc{M}_{s',w'}(\widecheck{\Phi})+\mathcal{D}_{s',w'}(\widecheck{\Phi}),
$$ 
where $\mathcal{D}_{s,w}(\Phi)$ is given by \eqref{D-def}.
\end{thm}

Theorem \ref{main} is a completely symmetrical formula but only holds when the real part of the parameters $s$, $w$, $s'$ and $w'$ are sufficiently large. In order to obtain a formula that also holds at the central point $s=w=s'=w'=\frac12$, we need to analytically continue the term $\mc{E}_{s,w}(\Phi)$. This is done in section \ref{analytic-continuation} under a technical condition enclosed in Hypothesis \ref{hypothesis-meromorphic}.

\subsection*{Spectral Reciprocity at the central point}

Let $\Pi$ be an everywhere unramified cuspidal automorphic representation for $\GL(3)$ over $F$. This means that $\Pi\simeq \otimes_v'\Pi_v$, where for each $v$, $\Pi_v$ is isomorphic to the isobaric sum
$$
|\cdot|_v^{it_{1,v}}\boxplus|\cdot|_v^{it_{2,v}}\boxplus|\cdot|_v^{it_{3,v}}.
$$
We say that $\Pi$ is $\theta$-tempered if for all $v$ and $i=1,2,3$, we have $|\Ree(t_{i,v})|\leq \theta$. It follows from a result of Luo, Rudnick and Sarnak \cite{LRS1999ramanujan} that every automorphic representation of $\GL(n)$ is $\theta$-tempered for some $\theta<1/2$. Therefore we can, and will, let $\theta=\theta(\Pi)<1/2$ be such that $\Pi$ is $\theta$-tempered.

Suppose that $\Phi_v$ is spherical for every archimedean place $v$. The reason for this restriction is twofold. The first and main one is because this leads to weight functions satisfying Hypothesis \ref{hypothesis-meromorphic}. The second one is that this trivializes the transformation $H_v\to \check{H}_v$ on the local archimedean weights. It would be very interesting to have a better understanding of this transformations. In particular it would be interesting to have an understanding of $\check{H}_v$ when is $H_v$ is taken to be a bump function selecting spectral parameters of a  certain size.  

Let $s, w\in \Cc$ let $\mf{q}$ and $\mf{l}$ be coprime ideals and write $U_{\infty}:=\prod_{v\mid \infty}\{y\in F_v^{\times};|y_v|=1\}$. Suppose that $\Phi=\Phi^{\mf{q},\mf{l}}$ is the choice of vector given in Section \ref{local-comp}. It follows from \eqref{Hv-for-l}, Proposition \ref{prop-local-comp} and Proposition \ref{Archimedean calculation}, that we have
$$ 
H(\pi)= \delta_{\infty}(\pi)\, \frac{\widehat{\lambda}_{\pi}(\mf{l},w)}{(N\mf{l})^w} \,\frac{\varphi(N\mathfrak{q})}{(N\mf{q})^2}h_{\mf{q}}(s,w;\Pi,\pi),
$$
where $\delta_{\infty}(\pi)$ is the characteristic function of representations that are unramified at every archimedean place, $\widehat{\lambda}_{\pi}(\mf{l},w)$ are modified Hecke eigenvalues given by
\begin{equation}\label{lambda-hat}
\widehat{\lambda}_{\pi}(\mf{l},w):=\sum_{\mf{a}\mf{b}=\mf{l}}\frac{\mu(a)}{(N\mf{a})^w}\lambda_{\pi}(\mf{b})=\prod_{\substack{\mf{p}^{n_{\mf{p}}}\mid\mid \mf{l}\\ n_{\mf{p}}\geq 1}}\left(\lambda_{\pi}(\mf{p}^{n_{\mf{p}}})-\frac{\lambda_{\pi}(\mf{p}^{n_{\mf{p}}-1})}{(N\mf{p})^w}\right),
\end{equation}
and, finally, $h_{\mf{q}}(s,w;\Pi,\pi)=1$ if $\operatorname{cond}(\pi)=\mf{q}$,  $h_{\mf{q}}(s,w;\Pi,\pi)\ll (N\mf{q})^{\theta+\epsilon}$ if $\operatorname{cond}(\pi)\mid \mathfrak{q}$ and vanishes otherwise. Thus, choosing $\Phi$ as above, we get

$$
\mc{M}_{s,w}(\Phi)=\mc{M}_0(\Pi,s,w,\mf{q},\mf{l}),
$$
where

$$
\mc{M}_0(\Pi,s,w,\mf{q},\mf{l}):=\mc{C}_0(\Pi,s,w,\mf{q},\mf{l})+\mc{E}_0(\Pi,s,w,\mf{q},\mf{l}),
$$
with
$$
\mc{C}_0(\Pi,s,w,\mf{q},\mf{l})=\frac{\varphi(N\mathfrak{q})}{(N\mf{q})^2}\sum_{\substack{\pi\text{ cusp}^0\\\operatorname{cond}(\pi)\mid \mf{q}}}\frac{\Lambda(s,\Pi\times \pi)\Lambda(w,\pi)}{\Lambda(1,\operatorname{Ad},\pi)}\frac{\widehat{\lambda}_{\pi}(\mf{l},w)}{(N\mf{l})^w}h_{\mf{q}}(s,w;\Pi,\pi)
$$
and
\begin{align*}
  \mc{E}_0(\Pi,s,w,\mf{q},\mf{l})=\frac{\varphi(N\mathfrak{q})}{(N\mf{q})^2}\sum_{\substack{\omega\in\widehat{F^{\times}U_{\infty}\bs\Aa^{\times}_{(1)}}\\\operatorname{cond}(\omega)^2\mid \mf{q}}}\int_{-\infty}^{\infty}\frac{\Lambda(s,\Pi\times \pi(\omega,it))\Lambda(w,\pi(\omega,it))}{\Lambda^{\ast}(1,\operatorname{Ad},\pi(\omega,it))}\\
\times \frac{\widehat{\lambda}_{\pi(\omega,it)}(\mf{l},w)}{(N\mf{l})^w}h_{\mf{q}}(s,w;\Pi,\pi(\omega,it))\frac{\d t}{2\pi},
\end{align*}
The notation $\text{cusp}^0$ denotes that we are restricting to forms that are unramified at every archimedean place and the analogous role in the Eisenstein part is played by quotienting by $U_{\infty}$. Finally, we let 
$$
\mc{N}_0(\Pi,s,w,\mf{q},\mf{l}):=\mathcal{D}_{s',w'}(\widecheck{\Phi})+\mc{R}_{s',w'}(\widecheck{\Phi})- \mathcal{D}_{s,w}(\Phi)-\mc{R}_{s,w}(\Phi),
$$
where $\mathcal{D}$ is given by \eqref{D-def} and $\mc{R}$ is given by \eqref{def-R}.

\begin{thm}\label{Spherical}
Let $\Pi$ be an everywhere unramified cuspidal automorphic representation of $\GL(3)$ over $F$. Suppose $\mf{q}$ and $\mf{l}$ are coprime ideals and that $\frac12\leq \Ree(s)\leq \Ree(w)< \frac34$. Then we have
$$
\mc{M}_0(\Pi,s,w,\mf{q},\mf{l})=\mc{N}_0(\Pi,s,w,\mf{q},\mf{l})+\mc{M}_0(\Pi,s',w',\mf{l},\mf{q}),
$$
where $s',w'$ are as in \eqref{s'w'}. Moreover, in this same region, $\mathcal{N}_0$ satisfies
\begin{equation}\label{ineq-for-N0}
  \mc{N}_0(\Pi,s,w,\mf{q},\mf{l})\ll_{s,w,\epsilon}\min(N\mathfrak{q},N\mathfrak{l})^{\theta-1+\epsilon}.
\end{equation}
\end{thm}

As an application, We may deduce a non-vanishing result which is similar in spirit to \citep[Theorem 1.2]{Khan2012Simultaneous}, where we prove an asymptotic formula for a family of forms of prime level $\mf{p}$ and let $N\mf{p}$ tend to infinity. It may be worth mentioning that although the results are similar, Khan's result concerns modular forms of sufficiently large weight $k$ for $\GL(2)$ over the field of rationals, while our result holds for everywhere unramified forms for $\GL(2)$ over an arbitrary number field.

%In the next corollary, we $\vartheta$ (which should not be confused with $\theta$ is an approximation towards the Ramanujan conjecture for $\GL(2)$, meaning that $\vartheta$ is such that $\pi_v$ is $\vartheta$-tempered for every local component of an automorphic representation of $\GL(2)$

\begin{cor}\label{non-vanishing}
Let $\Pi$ be an unramified cuspidal automorphic representation of $\GL(3)$ over $F$ and let $\mf{p}$ be a prime ideal of $\mf{o}_F$. Then for every $\epsilon>0$, we have

$$
\frac{\varphi(N\mf{p})}{(N\mf{p})^2}\sum_{\substack{\pi\text{ cusp}^0\\\operatorname{cond}(\pi)=\mf{p}}}\frac{\Lambda(\frac12,\Pi\times \pi)\Lambda(\frac12,\pi)}{\Lambda(1,\operatorname{Ad},\pi)}=\frac{4\Lambda(1,\Pi)\Lambda(0,\Pi)}{\xi_F(2)}+O_{\epsilon}\left((N\mf{p})^{\vartheta-\frac12+\epsilon}\right).
$$
In particular, for $N\mf{p}$ sufficiently large, there is at least one automorphic representation $\pi$ of conductor $\mf{p}$, unramified for every archimedean place and such that $\Lambda(\frac12,\Pi\times \pi)$ and $\Lambda(\frac12,\pi)$ are both non-zero.
\end{cor}

\subsection*{Plan of the paper}
In Section \ref{notations} we lay down our first conventions on number fields and local fields. In Section \ref{preliminaries} we recall the notion of automorphic representations for $\GL(n)$ over $F$ and some of its properties. Special attention is given to the case $n=2$ where, in particular, we recall the construcion of Eisenstein series and write down an explicit spectral decomposition. In Section \ref{Whittaker-section}, we introduce the Whittaker models and their relation to periods of Rankin-Selberg $L$-functions for $\GL(n+1)\times \GL(n)$. We work in complete generality but only use the results in the cases $n=1$ and $n=2$.

In Section \ref{Abstract-reciprocity} we prove an identity between periods which we call abstract reciprocity. This is connected to the actual reciprocity via a spectral decomposition which is performed in Section \ref{period-expansion}. In Section \ref{local-comp} we make some explicit computation for the local weights. Section \ref{degenerate-sec} is dedicated to analysing the degenerate term $\mathcal{D}_{s,w}(\Phi)$ and we show the meromorphic continuation of the spectral moment in Section \ref{analytic-continuation}, thus introducing the term $\mathcal{R}_{s,w}(\Phi)$. Theorem \ref{main} only uses the results up to Section \ref{period-expansion} and few observations from Section \ref{degenerate-sec}. On the other hand, Theorem \ref{Spherical} requires the full power of the results in Sections \ref{local-comp}, \ref{degenerate-sec} and \ref{analytic-continuation} and its proof is given in Section \ref{conclusion} along with that of Corollary \ref{non-vanishing}.

\section*{Acknowledgements}

This work benefited from discussions I had with many people including Philippe Michel, Rapha\"el Zacharias and Han Wu. I take this opportunity to thank them. I am also indebted to Subhajit Jana for pointing out a mistake in an earlier version of this paper. Finally, I thank the anonymous referee for his many remarks and corrections that greatly improved the quality of the text.
\section{Notations}\label{notations}

\subsection*{Number fields and completions}
Throughout the paper, $F$ will denote a fixed number field with ring of intergers $\mf{o}_F$ and discriminant $d_F$. For $v$ a place of $F$, we let $F_v$ be the completion of $F$ at the place $v$. If $v$ is non-archimedean, we write $\mf{o}_v$ for the ring of integers in $F_v$, $\mf{m}_v$ for its maximal ideal and $\varpi_v$ for its uniformizer. The adele ring of $F$ is denoted by $\Aa$, its unit group is denoted by $\Aa^{\times}$ and finally, $\Aa_{(1)}^{\times}$ denotes the ideles of norm $1$. We also fix once and for all, an isomorphism $\Aa^{\times}\simeq \Aa_{(1)}^{\times} \times \Rr_{>0}$. %This last thing amounts to fixing an isometric embedding $\Rr_{>0}\xhookrightarrow{} \Aa^{\times}$.

\subsection*{Additive characters} We let $\psi =\otimes_v\psi_v$ be the additive character $\psi=\psi_{\Qq}\circ\operatorname{Tr}_{F/\Qq}$ where $\operatorname{Tr}_{F/\Qq}$ it the trace map and $\psi_{\Qq}$ is the additive character on $\Aa_{\Qq}$ which is trivial on $\Qq$ and such that $\psi(x)=e^{2\pi i x}$ for $x\in \Rr$. Let $d_v$ be the conductor of $\psi_v$, \textit{i.e.}: the smallest non-negative integer such that $\psi_v$ is trivial on $\mf{m}_v^{d_v}$. Notice that $d_v=0$ for every finite place not dividing the discriminant and we have the relation $d_F=\prod_vp_v^{d_v}$, where $p_v:=|\mf{o}_v/\mf{m}_v|$.

\subsection*{Measures} In the group $\Aa$ we use a product measure $\d x=\prod_v\d x_v$, where for real $v$, $\d x_v$ is the Lebesgue measure on $\Rr$, for complex $v$, $\d x_v$ is twice the Lebesgue measure on $\Cc$ and for each finite $v$, $\d x_v$ is a Haar measure on $F_v$ giving measure $p^{-d_v/2}$ to the compact subgroup $\mf{o}_v$. As for the multiplicative group $\Aa^{\times}$, we also take a product measure $\d^{\times}x=\prod \d^{\times} x$, where $\d^{\times} x_v=\zeta_v(1)\frac{\d x_v}{|x_v|}$ for infinite or unramified $v$ and we take $\d^{\times}x_v:=p_v^{d_v/2}\xi_{F_v}(1)\frac{\d x_v}{|x_v|}$, for ramified $v$ so that for any finite $v$, we are giving measure $1$ to $\mathfrak{o}^{\times}_v$. Such measures can naturally give rise to measures on the quotient spaces $F\backslash\Aa$ and $F^{\times}\backslash\Aa^{\times}_{(1)}$ such that

$$
\operatorname{vol}(F\backslash \Aa)=1\text{, and }\operatorname{vol}(F^{\times}\backslash\Aa^{\times}_{(1)})=d_F^{1/2}\xi^{\ast}_F(1).
$$

The first can be found in Tate's thesis (see \cite[Chapter XV]{cassels-frohlich}) and the second is \cite[Proposition XIV.13]{Lang2013algebraic} (the factor $d_F^{1/2}$ comes from our different normalization of the multiplicative measure).

%\subsection{Subgroups of $\GL(2$)}
%Let $R$ be a commutative ring. We will consider the following subgroups of $\GL_2(R)$:
%$$
%\mathrm{Z}_2(R):=\left\{z(u),\;u\in R^{\times}\right\},\;\mathrm{N}_2(R):=\left\{n(x);\;x\in R\right\},
%$$
%$$
%\mathrm{A}_2(R):=\left\{a(y);\;y\in R^{\times}\right\},\;B_2(R)=Z_2(R)N_2(A_2)A_2(R)
%$$
%where
%%
%$$
%z(u):=\mtrx{u&\\&u},\;\;n(x):=\mtrx{1&x\\&1},\;\;a(y):=\mtrx{y&\\&1}.
%$$

\section{Preliminaries on automorphic representations}\label{preliminaries}

In the course of studying automorphic forms in $\GL(n)$, it will be important to distinguish a few of its subgroups. For any unitary ring $R$ with group of invertible elements given by $R^{\times}$, we let $Z_n(R)$ denote the group of central matrices (\textit{i.e.} non-trivial multiples of the identity), $N_n(R)$ denote the maximal unipotent group formed by matrices with entries $1$ on the diagonal and $0$ below the diagonal and we let $A_n(R)$ denote the diagonal matrices with lower-right entry $1$.
%In other words,
%
%
%$$
%Z_n(R):=\left\{\mtrx{u&&&\\&u&&\\ &&\ddots &\\&&& u};\;u\in R^{\ast}\right\},
%$$
%
%$$
%N_n(R):=\left\{\mtrx{1&x_{1,2}&&&x_{1,n}\\&1&x_{2,3}\\&&\ddots &\ddots\\ &&&1&x_{n-1,n} \\&&&&1};\;x_{i,j}\in R,\,\, 1\leq i<j\leq n\right\},
%$$
%and
%$$
%A_n(R):=\left\{\mtrx{y_{n-1}\ldots y_1&&&&\\&y_{n-2}\ldots y_1&&&\\ &&\ddots &&\\&&& y_1 \\&&&&1};\;y_i\in R^{\times},\,\, 1\leq i\leq {n-1}\right\}.
%$$
%

We extend our additives character to $N_n$ in the following way: If  $n=(x_{i,j})_{1\leq i,j\leq n}\in\N_n(\Aa)$,
%is such that the entries right above the diagonal are $x_{1,2},\ldots,x_{n-1,n}$
then $\psi(n):=\psi(x_{1,2}+\ldots+x_{n-1,n})$ and similarly for $\psi_v$.
We can extend the measures on the local fields $F_v$ and their unit groups $F_v^{\times}$ to measures on the groups $Z_n(F_v)$, $N_n(F_v)$ and $A_n(F_v)$ using the obvious isomorphisms $Z_n(R)\simeq R^{\times}$, $N_n(R)\simeq R^{\frac{n(n-1)}{2}}$ and $A_n(R)\simeq (R^{\times})^{n-1}$.
%coordinates used in their definitions.
%In our work, we only need to integrate on subgroups of $\GL_2$, we shall define these measures only in this case.

Moreover, let $K_v$ denote a maximal compact subgroup of $\GL_n(F_v)$ given by

$$
K_v:=
\begin{cases}
O(n),\text{ if }F_v=\Rr,\\
U(n),\text{ if }F_v=\Cc,\\
\GL_n(\mf{o}_v),v<\infty.
\end{cases}
$$

We can now define a Haar measure on $\GL_n(F_v)$ by appealing to the Iwasawa decomposition. Let $\d k$ be a Haar probability measure on $K_v$ and consider the surjective map

\begin{align*}
Z_n(F_v)\times N_n(F_v)\times A_n(F_v)\times K_v&\rightarrow \GL_n(F_v),\\
(z,n,a,k)&\mapsto znak
\end{align*}
and let $\d g_v$ be the pull-back by this map of the measure

$$
\Delta(a)^{-1}\prod_{k=1}^{n-1}y_k^{-k(n-k)} \times \d z \times \d n\times \d a\times \d k,
$$
where
$$
\Delta\begin{pmatrix} y_1&&&\\&\ddots&&\\&&y_{n-1}&\\&&&1 \end{pmatrix}=\prod_{j=1}^{n-1} |y_j|^{n+1-2j}.
$$
In particular, for $\GL_2$,
%let

%As $u,y$ run through $F_v^{\times}$ and $x$ through $F_v$, these functions give parametrizations of the groups $Z_2(F_v)$, $N_2(F_v)$ and $A_2(F_v)$. Thus we can give them the same measures as those of $F_v^{\times}$ and $F_v$. Then for every measurable function on $\GL_2(F_v)$, we have

$$
\int_{\GL_2(F_v)}f(g_v)\d g_v=\int_{K_v}\int_{F_v^{\times}}\int_{F_v}\int_{F_v^{\times}} f\left(z(u)n(x)a(y)k\right)\d^{\times} u\d x\frac{\d^{\times} y}{|y|_v}\d k,
$$
where
$$
z(u)=\mtrx{u&\\&u},\;n(x)=\mtrx{1&x\\&1},\;a(y)=\mtrx{y&\\&1}.
$$

Similarly, we shall consider measures on the quotients $N_n(F_v)\bs \GL_n(F_v)$ and $\operatorname{PGL}_n(F_v):=Z_n(F_v)\bs \GL_n(F_v)$ by omiting the terms $\d n$ and $\d z$ respectively. Now, given a group $G$ for which we have attached Haar measures $\d g_v$ to $G(F_v)$, we attach to $G(\Aa)$ the product measure $\d g=\prod_v\d g_v$. Since $\PGL_2(F)\hookrightarrow \PGL_2(\Aa)$ discretely, we may use the measure of $\PGL_2(\Aa)$ to define one on
$$
\mathbf{X}:=\PGL_2(F) \bs \PGL_2(\Aa)=Z_2(\Aa)\GL_2(F)\bs \GL_2(\Aa),
$$
which turns out to have finite total measure $\operatorname{vol}(\mathbf{X})<+\infty$.

\subsection{Automorphic representations for $\GL(2)$}

Consider the Hilbert space $L^2(\mathbf{X})$ with an action of $\GL_2(\Aa)$ given by right multiplication and a $\GL_2(\Aa)$-invariant inner product given by
\begin{equation}\label{IP-L2}
\IP{\phi_1,\phi_2}_{L^2(\mathbf{X})}=\int_{\mathbf{X}}\phi_1(g)\overline{\phi_2(g)}\d g.
\end{equation}
It is well-known that this space decomposes as
\begin{equation}\label{spec-decomp}
L^2(\mathbf{X})=L^2_{\text{cusp}}(\mathbf{X})\oplus L^2_{\text{res}}(\mathbf{X})\oplus L^2_{\text{cont}}(\mathbf{X}),
\end{equation}
where $L^2_{\text{cusp}}(\mathbf{X})$ denotes the closed subspace of cuspidal functions given by the functions satisfying the relation

$$
\int_{N_2(F)\bs N_2(\Aa)}\phi(ng)\d n=0,
$$
$L^2_{\text{res}}(\mathbf{X})$ is the residual spectrum consisting of all the one-dimensional subrepresentations of $L^2(\mathbf{X})$ and $L^2_{\text{cont}}(\mathbf{X})$ is expressed in terms of Eisenstein series which we discuss further below. Moreover, $L^2_{\text{cusp}}(\mathbf{X})$ decomposes as a direct sum of irreducible representations, which are called the cuspidal automorphic representations.

\subsubsection{Induced representations and Eisenstein series}\label{Ind-Eis}

Given a (not necessarily unitary) character $\omega: F^{\times}\backslash \Aa^{\times}\rightarrow \Cc^{\times}$, we denote by $\pi(\omega)$ the isobaric sum $\omega\boxplus \omega^{-1}$, \textit{i.e.} the space of \emph{measurable} functions $f$ on $\GL_2(\Aa)$ such that

$$
f\left(\mtrx{a&b\\&d}g\right)=|a/d|^{1/2}\omega(a)\omega^{-1}(d)f(g),\; \IP{f,f}_{\text{Ind}}<+\infty,
$$
where $|\cdot|$ denotes the adelic norm, $K:=\prod_{v}K_v$ and $\IP{f_1,f_2}_{\text{Ind}}<\infty$, where we put
\begin{align}\label{Ind-IP}
\IP{f_1,f_2}_{\text{Ind}}&:=\int_{F^{\times}\backslash\Aa^{\times}_{(1)}\times K} f_1(a(y)k)\overline{f_2(a(y)k)}\d^{\times}y\,\d k\notag\\
&=\operatorname{vol}(F^{\times}\backslash\Aa^{\times}_{(1)})\int_K f_1(k)\overline{f_2(k)}\d k.
\end{align}

Given such $\omega$ and $f\in \pi(\omega)$, we define an Eisenstein series by a process of analytic continuation. It is given by the following series, as long as it converges:

$$
Eis(f)(g):=\sum_{\gamma\in B_2(F)\bs\GL_2(F)}f(\gamma g),
$$
where for a ring $R$, 
$$
B_2(R):=\left\{\mtrx{a&b\\&d};\;a,d\in R^{\times},\;b\in R\right\}.
$$
Suppose $\omega\neq 1$. We define further the \emph{normalized Eisenstein series} by taking, 
$$
\widetilde{Eis}(f):=L(1,\omega^2)Eis(f).
$$

It will be convenient to define the inner product of two normalized Eisenstein series in terms of the inner product in the induced model of the functions used for generating it. In other words, for $f_1,f_2\in\pi(\omega)$ and $\phi_i=\widetilde{Eis}(f_i)$, $i=1,2$, we write

\begin{equation}\label{Eis-IP}
  \IP{\phi_1,\phi_2}_{\widetilde{Eis}}:=|L(1,\omega^2)|^2\IP{f_1,f_2}_{\text{Ind}}=d_F^{1/2}\Lambda^{\ast}(1,\operatorname{Ad}\pi(\omega))\int_{K}f_1(k)\overline{f_2(k)}\d k.
\end{equation}

Finally, for a complex parameter $s$, we use the notation $\pi(\omega,s):=\pi(\omega\times|\cdot|^{s})$. For a character $\omega_v$ of $F_v$ we can similarly define the induced representation $\pi_v(\omega_v,s)$ so that if $\omega\simeq \otimes'_v \omega_v$, we have $\pi(\omega,s)\simeq \otimes_v'\pi_v(\omega_v,s)$. %Also for $f\in\pi(\omega)$, we let $f(s)$ to be the unique function on $\pi(\omega,s)$ which coincides with $f$ when restricted to $K$ and $E(s,f):=Eis(f(s))$.

\subsubsection{Spectral decomposition for smooth functions}

We already encountered a decomposition of $L^2(\mathbf{X})$ in \eqref{spec-decomp}, but in practice we will encounter functions in $L^2(\mathbf{X})$ which are right invariant by a large compact subgroup $K_0\subset \GL_2(\Aa)$  and moreover we will need more uniformity than simply $L^2$-convergence. In the following, we write down a more precise form of this decomposition for functions in $\mc{C}^{\infty}(\mathbf{X}/K^S)$, where for a finite set $S$ of places of $F$ containing the archimedean ones, $K^S$ is the compact group given by

$$
K^S:=\prod_{v\not\in S}K_v.
$$

The only intervening representations are those that are unramified outside $S$. That means $\pi\in C(S)$ or $\pi=\pi(\omega,it)$, $\pi=\omega\circ\det$  for $\omega\in \Xi(S)$.  For each cuspidal automorphic representation $\pi$, we let $\mc{B}_{c}(\pi)$ denote an orthonormal basis of the realization of $\pi$ in $L^2(\mathbf{X})$ with respect to the inner product $\IP{\cdot,\cdot}_{L^2(\mathbf{X})}$. Similiarly, for an induced representation $\pi=\pi(\omega)$, we define $\mathcal{B}_e(\pi)$ to be a basis of normalized Eisenstein series (not vectors in the induced models!) with respect to $\IP{\cdot,\cdot}_{\widetilde{Eis}}$. We may therefore state the following version of the spectral theorem:

\begin{prop}\label{spectral theorem} Let $F\in \mc{C}^{\infty}\left(\mathbf{X}/K^S\right)$ and of rapid decay, then

\begin{align*}
F(g) = \sum_{\pi\in C(S)}\sum_{\phi\in\mc{B}_c\left(\pi\right)}\IP{ F, \phi}\phi(g) +\operatorname{vol}(\mathbf{X})^{-1}\sum_{\substack{\omega\in\Xi(S)\\ \omega^2=1}}\IP{F,\omega\circ\det}\omega(\det g)\\
+\frac{1}{4\pi}\sum_{\omega\in \Xi(S)}\int_{-\infty}^{\infty}\sum_{\substack{\phi\in\mc{B}_e\left(\pi(\omega,it)\right)}}\IP{F,\phi}\phi(g)\d t,
\end{align*}
where $\IP{,}$ denotes the same integral as in the definition of $\IP{,}_{L^2(\mathbf{X})}$ and convergence is absolute and uniform for $g$ on any compact subset of $\mathbf{X}$.
\end{prop}

The result, for pseudo Eisenstein series. follows from equations (4.21) and (4.25) in \cite{GelbartJacquet1979forms} and by extending the inner product $(a_1(iy),a_2(iy))$ with respect to an orthogonal basis of $L^2(F^{\times}\backslash \Aa_{(1)}^{\times}\times K)$. The general result is a consequence of the fact that the space of cusp forms decomposes discretely and spans the orthogonal complement to the space of pseudo Eisenstein series. 

\section{Whittaker models and periods}\label{Whittaker-section}

In this section, we consider irreducible automorphic representations $\pi$ of $\GL_n(\Aa)$ and the period integrals related to some Rankin-Selberg $L$-functions. We will only be concerned with the generic representations, which are those admitting a Whittaker model. This is done for arbitray $n$ but only the cases $n=2$ and $n=1$ are used in the sequel. Finally, we shall not distinguish between a representation $\pi$ and its space of smooth vectors $V_{\pi}^{\infty}$. An automorphic form $\phi$ will always denote a smooth vector in an irreducible automorphic representation.

\subsection{Whittaker functions}\label{whittaker-decomposition}
Let $\pi$ be a generic automorphic representation of $\GL_n(\Aa)$ and let $\phi\in \pi$ be an automorphic form. Let $W_{\phi}:\GL_n(\Aa)\rightarrow \Cc$ be the Whittaker function of $\phi$ given by
\begin{equation}\label{Whittaker-Integral}
  W_{\phi}(g)=\int_{N_n(F)\bs N_n(\Aa)}\phi\left(n g\right)\overline{\psi(n)}\d n.
\end{equation}
%with convergence absolute and uniform on compact subsets. 
It satisfies $W_{\phi}(ng)=\psi(n)W_{\phi}(g)$ for all $n\in\N_n(\Aa)$.
%\begin{thm}
%$$
%\phi(g)=\sum_{\gamma\in N_{n-1}(F)\bs G_{n-1}(F)}W_{\phi}\left(\mtrx{\gamma&\\&1}g\right).
%$$
%\end{thm}

Given a cuspidal automorphic representation of $\GL_n(\Aa)$ we might write down an isomorphism $\pi\simeq\otimes'_v\pi_v$ where for each $v$, $\pi_v$ is a local generic admissible representation of $\GL_n(F_v)$ and we might define Whittaker functions for each local representation such that for every $\phi\in \pi$ with $\phi=\otimes'_v\phi_v$ through the above isomorphism, we have

\begin{equation}\label{Whittaker product}
W_{\phi}(g)=\prod_vW_{\phi_v}(g_v),\;g=(g_v)_v\in \GL_n(\Aa).
\end{equation}
%In fact the definition of a Whittaker function makes sense if we replace $\psi$ (or $\psi_v$) by any other non-trivial additive character, but in order to lighten notation we decided to define it only for $\psi$ as defined in the section \ref{notations}. 
%
In fact, the map $\phi\mapsto W_{\phi}$ is an intertwiner between $\pi$ and its image, denoted by $\mathcal{W}(\pi,\psi)$, the so-called Whittaker model of $\pi$. We similarly define the local Whittaker models $\mc{W}(\pi_v,\psi_v)$. Later on, it will be convenient to exchange freely between a representation and its associated Whittaker model. The importance of the latter comes from its close relation to local Rankin-Selberg $L$-functions, as we will see in \S \ref{integral-rep}.

There is a similar story for non-cuspidal forms but in this case it is better to work with \emph{normalized} Eisenstein series. As we will only need this for $n=2$, we shall restrict to this case. Let $f\in \pi(\omega)$ and supposes that $f$ is factorable, \emph{i.e.} $f=\otimes_v'f_v$ with $f_v\in\pi_v(\omega_v)$. Then, it follows by analytic continuation and Bruhat decomposition that
\begin{align*}
  W_{\widetilde{Eis}(f)}(g)=L(1,\omega^2)\int_{N_2(\Aa)}f(wng)\overline{\psi(n)}\d n=\prod_v  W^{J}_{f_v}(g_v),
\end{align*}
where $W^J_{f_v}$ is the normalized Jacquet integral, given by
$$
W^J_{f_v}(g_v)=L(1,\omega_v^2)\int_{N_2(F_v)}f_v(wng)\overline{\psi(n)}\d n.
$$
By putting $\phi=\widetilde{Eis}(f)$ and $W_{\phi_v}:=W^{J}_{f_v}$, we see that \eqref{Whittaker-product} also holds in this case.

It is also important to consider Whittaker functions with respect to the inverse character $\psi'=\overline{\psi}$, so we analogously define $W'_{\phi}$ and $W'_{\phi_v}$ by replacing $\psi_v$ by $\psi'_v=\overline{\psi_v}$ and $\psi$ by $\psi'=\prod_v\psi'_v$ in all the previous definitions. It follows from uniqueness of local Whittaker functions that we  may take

\begin{equation}\label{bars-and-primes}
W'_{\overline{\phi_v}}=\overline{W_{\phi_v}}\text{ for all places $v$ of $F$}.
\end{equation}

If a local generic admissible representation $\pi_v$ is unramified for some finite place $v$, this mean that in $\pi_v$ there exists a vector which is right invariant by the action of $\GL_n(\mf{o}_v)$. Such a vector is called \textit{spherical} and spherical vectors are unique up to multiplication by scalars. Among the spherical vectors we shall distinguish a certain one which we call \textit{normalized spherical}. If $v$ is unramified, the normalized spherical vector will be the one for which $W_{\phi_v}(e)=1$, where $e\in\GL_n(F_v)$ denotes the identity element. For the finite ramified places we simply define it to be the newform (defined in \S \ref{newforms}). This avoids repetition and is justified by the fact that the two notions also coincide for unramified primes.

\subsection{Integral representations of $\GL_{n+1}\times \GL_n$ $L$-functions}\label{integral-rep}

This theory is an outgrowth of Hecke's theory of $L$-functions for $\GL_2$ and has been developed by Jacquet, Piatetski-Shapiro and Shalika. We start with $\Pi$ and $\pi$ irreducible automorphic representations of $\GL_{n+1}(\Aa)$ and $\GL_n(\Aa)$ respectively and let $\Phi\in\Pi$, $\phi\in \pi$ be automorphic forms. Suppose momentarily that $\Phi$ is a cusp form and hence rapidly decreasing. We can thus consider for every $s\in \Cc$ the following integral
$$
I(s,\Phi,\phi):=\int_{\GL_n(F)\bs \GL_n(\Aa)}\Phi\mtrx{h&\\&1}\phi(h)|\det h|^{s-\frac12}\d h.
$$
It follows from the Whittaker decomposition of cusp forms (see \cite[Theorem 1.1]{Cogdell2007functions}) that if $\Phi$ is a cuspidal function
\begin{equation}\label{I=W-cusp}
I(s,\Phi,\phi)=\Psi(s,W_{\Phi},W'_{\phi}),\;\;(\Ree(s)\gg 1),
\end{equation}
where
\begin{equation}\label{Psi}
\Psi(s,W,W'):=\int_{N_n(\Aa)\bs \GL_n(\Aa)}W\mtrx{h&\\&1}W'(h)|\det h|^{s-\frac12}\d h.
\end{equation}
Our next result gives some of the good properties of $\Psi(s,W,W')$, \textit{i.e.} convergence, boundedness in vertical strips and the fact that it factors into local integrals whenever $\Phi$ and $\phi$ also factor.

  \begin{prop}\label{RS} Let $\Pi$ and $\pi$ be automorphic representations of $\GL(n+1)$ and $\GL(n)$ over $F$, respectively. Let $\Phi=\otimes'_v\Phi_v\in \Pi$ and $\phi=\otimes'_v\phi_v\in \pi$ be automorphic forms. Let $W_{\Phi_v}$ and $W'_{\phi_v}$ be as in \S  \ref{whittaker-decomposition}. Then, for $\Ree(s)\gg 1$, $\Psi(s,W_{\Phi},W'_{\phi})$ converges and we have the factorization
$$
\Psi(s,W_{\Phi},W'_{\phi})=\prod_v\Psi_v(s,W_{\Phi_v},W'_{\phi_v}),
$$
where
\begin{equation}\label{Psiv}
\Psi_v(s,W,W'):=\int_{N_n(F_v)\bs \GL_n(F_v)}W\mtrx{h_v&\\&1}W'(h_v)|\det h_v|^{s-\frac12} \d h_v.
\end{equation}
Moreover, if $v$ is finite and both $\Pi_v$ and $\pi_v$ are unramified and $\Phi_v$ and $\phi_v$ are normalized spherical,
$$
\Psi_v(s,W_{\Phi_v},W'_{\phi_v})=p_v^{d_vl_n(s)}L(s,\Pi_v\times\pi_v),
$$
where
$$
l_n(s)=\frac{n(n+1)}{2}s-\frac{n(n+1)(2n+1)}{12}.
$$
\end{prop}

\begin{proof}
  The first part follows from gauge estimates for Whittaker functions (\textit{cf.} \cite[\S 2]{JPSS1979automorphic}).  It is an important fact that this part does not require the representations to be cuspidal. The reason is that in some sense, the integral representation using Whittaker functions only sees the non-constant terms. For the local computation this is well-known when $p_v$ is unramified (see \emph{e.g.} \cite[Theorem 3.3]{Cogdell2007functions}). In general we may resctrict to the unramified situation by following the computation in the proof of \cite[Lemma 2.1]{CPS1994converse}.
\end{proof}

\begin{rmk}\label{shorter-notation}
In the case $n=1$ we shall write $I(s,\phi)$, $\Psi(s,W_{\phi})$ and $\Psi_v(s,W_{\phi_v})$ instead of $I(s,\phi,\mathbf{1})$, $\Psi(s,W_{\phi},W_{\mathbf{1}_v})$ and $\Psi_v(s,W_{\phi_v},W_{\mathbf{1}})$, where $\mathbf{1}$ and $\mathbf{1}_v$ denote the constant functions on $\GL_1(\Aa)$ and $\GL_1(F_v)$ respectively. %We note that in this case many conditions become trivial. For example, cuspidality becomes automatic and $W_{\mathbf{1}}$ is the same as $\mathbf{1}$.
\end{rmk}

\subsection{Newforms and ramified $L$-factors}\label{newforms}

For a finite place $v$ and any admissible irreducible generic representation of $\GL_n(F_v)$, not necessarily unramified, we define a distinguished vector in its Whittaker model, called \textit{newform}. This was first introduced by Casselman \cite{Casselman1973some} when $n=2$ by translating the results of Atkin--Lehner to the representation-theoretic language. This was later generalized by Jacquet--Piatetski-Shapiro-Shalika \cite{JPSS1981conducteur} for general $n$ by requiring that they are good test vectors for representating $L$-functions via Rankin-Selberg periods as in \S \ref{integral-rep}. Moreover, when $\pi_v$ is unramified, these coincide with normalized spherical vectors.

The fact that these newvectors are test vectors for Rankin-Selberg $L$-functions can be rephrased by relating its values to the Langlands parameters of the representation. This was carefully carried out in \cite{Miyauchi2014Whittaker}. In order to quote these results we introduce the following pieces of notation: For $\underline{\nu}=(\nu_1,\ldots,\nu_{n-1})\in \mathbb{Z}^{n-1}$, let $s(\underline{\nu})=\sum_{i=1}^{n-1}\frac{i(n-i)\nu_i}{2}$ and for $y\in F_v^{\times}$, we write 
$$
a(\underline{\nu}):=\begin{pmatrix} \varpi_v^{\nu_1+\ldots+\nu_{n-1}}&&&\\&\varpi_v^{\nu_2+\ldots+\nu_{n-1}}&&\\&&\ddots&&\\ &&&1 \end{pmatrix}.
$$
The main result of \cite{Miyauchi2014Whittaker} states that if $\psi_v$ is unramified, we have
$
W_{\pi_v}(a(\underline{\nu}))=p^{s(\nu)}\lambda_{\pi}(\underline{\nu}),
$
where
%
%\begin{equation}\label{Whit-coef}
%  \lambda_{\pi_n}(\nu_1,\ldots,\nu_{n-1}):=p_v^{\sum_{i=1}^{n-1}\frac{i(n-i)}{2}\nu_i}W_{\pi_v}\begin{pmatrix} \varpi_v^{\nu_1+\ldots+\nu_{n-1}}&&&\\&\varpi_v^{\nu_2+\ldots+\nu_{n-1}}&&\\&&\ddots&&\\ &&&1 \end{pmatrix}
%\end{equation}
%
\begin{equation}\label{only-small}
\lambda_{\pi_v}(\underline{\nu})=0,\text{ unless }\nu_1,\ldots,\nu_{n-1}\geq 0,
\end{equation}
and are in general given by Schur polynomials evaluated on the Langlands parameters of $\pi_v$ (\textit{cf.} \cite{Miyauchi2014Whittaker} for details).

Now, when $\psi_v$ is ramified, this has to be modified. First, we write $\psi_v(x)=\psi_{F_v}(\lambda x)$ for some $\lambda\in F_v^{\times}$, where $\psi_{F_v}$ is an unramified additive character of $F_v$ and let $d=v(\lambda)$. We then define the \emph{newvector} by taking
$$
W_{\pi_v}(g)=W^{unr}_{\pi_v}(a(\iota_n(d))g),
$$
where $\iota_n(d)=(d,d,\ldots,d)\in \Zz^{n-1}$ and $W^{unr}_{\pi_v}$ denotes the newvector for the unramified character $\psi_{F_v}$. The term $a(\iota_n(d))$ is responsible by the change in the additive character.

In addition to Proposition \ref{RS} we shall also need to compute $L$-functions for certain ramified local representations. In particular, we require the following computation that appears, for instance in the work of Booker-Krishnamurhty-Lee (see the proof of \cite[Lemma 3.1]{BKL2020test}): Let $n>m$ and let $\Pi_v$ (resp. $\pi_v$) be an irreducible admissible generic representation of $\GL_n(F_v)$ (resp. $\GL_m(F_v)$) with Langlands parameters $(\gamma_{\Pi_v}^{(i)})_{i=1}^n$ (resp. $(\gamma_{\pi_v}^{(j)})_{j=1}^m$ ). Suppose further that $\pi_v$ is ramified, then one has

\begin{equation}\label{L-factor}
\sum_{\substack{\underline{\nu}\in \mathbb{Z}^{m-1}\\ \nu_1,\ldots,\nu_{m-1}\geq 0}}\frac{\lambda_{\Pi_v}(\underline{\nu},0,\ldots,0)\lambda_{\pi_v}(\underline{\nu})}{p_v^{(\nu_1+2\nu_2+\ldots+(m-1)\nu_{m-1})s}}=L(s,\Pi_v\times\pi_v),
\end{equation}
where 
\begin{equation}\label{L-naive}
  L(s,\Pi_v\times \pi_v):=\prod_{\substack{1\leq i\leq n\\1\leq j\leq m}}\left(1-\gamma_{\Pi_v}^{(i)}\gamma_{\pi_v}^{(j)}p_v^{-s}\right)^{-1}.
\end{equation}
This coincides with Langlands local $L$-function when $\Pi_v$ is unramified, which we shall suppose.
\subsection{Relation between inner products on $\GL(2)$}

Let $\pi$ be a generic automorphic representation of $\GL(2)$ over $F$ with trivial central character. We define a $\GL_2(\Aa)$ invariant inner product on the representation space of $\pi$ as follows: If $\pi$ is cuspidal, then we may see $V_{\pi}$ embedded in $L^2(\mathbf{X})$ and therefore $\pi$ may inherit the inner product from $L^2(\mathbf{X})$ given by \eqref{IP-L2}. If $\pi$ is Eisenstein we cannot see the representation space of $\pi$ inside $L^2(\mathbf{X})$ and hence we equip it with the inner product given by \eqref{Eis-IP}.

There is however another way of defining an inner product for factorable vectors in these representations which is independent of whether $\pi$ is cuspidal or Eisentein. This is done by using the Whittaker model as follows: For each place $v$, we have a $\GL_2(F_v)-$invariant inner form on $\mc{W}(\pi_v,\psi_v)$ by letting.
\begin{equation}\label{whitt-IP}
  \vartheta_v(W_1,W_2)=\frac{\int_{F_v^{\times}}W_1(a(y_v))\overline{W_2(a(y_v))}\d ^{\times} y_v,}{\zeta_v(1)L_v(1,\operatorname{Ad}\pi)/\zeta_v(2)}.
\end{equation}
The fact that the numerator of \eqref{whitt-IP} is indeed right $\GL_2(F_v)-$invariant follows form the theory of the Kirillov model and the inclusion of the denominator is to ensure the following property: Whenever $\pi_v$ and $\psi_v$ are unramified and $W$ is normalized spherical, we have $\vartheta_v(W,W)=1$. Finally, let $\phi_1=\otimes\phi_{1,v}$ and $\phi_2=\otimes\phi_{2,v}$ be either cusp forms or normalized Eisenstein series, we define the \textit{canonical} inner product by the formula
\begin{equation}\label{Whittaker-product}
  \IP{\phi_1,\phi_2}_{can}:=2d_F^{1/2}\Lambda^{\ast}(1,\operatorname{Ad}\pi)\times \prod_v\vartheta_v(W_{\phi_{1,v}},W_{\phi_{2,v}}).
\end{equation}
Since every two $\GL_2(\mathbb{A})-$invariant inner products in $\pi$ must be equal up to multiplication by some scalar, it follows that we can compare the canonical inner product with the ones introduced earlier for cuspidal and Eisenstein representations. Indeed, Rankin-Selberg theory in the cuspidal case and a direct computation in the Eisenstein case gives us the following relation: 
\begin{equation}\label{relation between inner products}
\IP{\phi_1,\phi_2}_{can}=
\begin{cases}
\IP{\phi_1,\phi_2}_{L^2(\mathbf{X})},\text {if }\pi\text{ is cuspidal},\\
2\IP{\phi_1,\phi_2}_{\text{Eis}},\text {if }\pi\text{ is Eisenstein}.
\end{cases}
\end{equation}

The computation in the Eisenstein case follows from \cite[Lemma 2.8]{Wu2014Burgess}. For the cusp forms we combine the proof of \cite[Proposition 2.13]{Wu2014Burgess} with the value of the residue of an Eisenstein series computed in \cite[Eq. (4.6)]{MV2010subconvexity}\footnote{In \cite{Wu2014Burgess}, a factor $d_F^{1/2}$ seems to be missing in the computation of this residue.}.

\section{Abstract reciprocity}\label{Abstract-reciprocity}

In this section we show an identity between two periods. At this point we make no attempt to relate them to moments of $L-$funcitons. The proof is a rather simple matrix computation.

Suppose $\Phi\in \mc{C}^{\infty}(Z_3(\Aa)\GL_3(F)\bs\GL_3(\Aa))$ is such that for every $h\in \GL_2(\mathbb{A})$, the integral

\begin{equation}\label{AsPhi}
\mathcal{A}_s\Phi(h):=|\det h|^{s-\frac12}\int_{F^\times\bs\Aa^{\times}}\Phi\mtrx{z(u)h&\\&1}|u|^{2s-1}d^{\times} u.
\end{equation}
converges and such that $y\mapsto \mc{A}_s\Phi\left(\mtrx{y&\\&1}h\right)$ is of rapid decay as $|y|$ $\rightarrow$ $0$ or $+\infty$.

\begin{prop}\label{abstract-reciprocity}

  Let $\Phi$ be as above and let $I(w,\cdot)$ be as in Remark \ref{shorter-notation}. Then, for every $s,\,w\in\Cc$, we have the reciprocity relation

$$
I(w,\mc{A}_s\Phi)=I(w',\mc{A}_{s'}\widecheck{\Phi}),
$$
where $(s',w')$ are as in \eqref{s'w'} and 
\begin{equation}\label{check-Phi}
  \widecheck{\Phi}(g):=\Phi(gw_{23}),\;\;w_{23}=\mtrx{1&&\\&&1\\&1&}.
\end{equation}

\end{prop}
\begin{proof}
By definition,

\begin{equation}\label{LwAs=intint}
I(w,\mc{A}_s\Phi)=\int_{F^{\times}\bs\Aa^{\times}}\int_{F^{\times}\bs\Aa^{\times}}\Phi\mtrx{z(u)a(y)&\\&1}|u|^{2s-1}|y|^{s+w-1}\d^{\times}u\,\d^{\times}y.
\end{equation}

Now, since $\Phi$ is left invariant by $Z_3(\Aa)\GL_3(F)$, we see that for every $u,\;y\in\Aa^{\times}$, one has

\begin{align*}
\Phi\mtrx{uy&&\\&u&&\\&&1}&=\Phi\left(\mtrx{u&&\\&u&\\&&u}w_{23}\mtrx{y&&\\&
u^{-1}&\\&&1}w_{23}\right)\\
&=\Phi\left(\mtrx{y&&\\&u^{-1}&\\&&1}w_{23}\right)=\check{\Phi}\mtrx{y&&\\&u^{-1}&\\&&1}.
\end{align*}

%By a completely similar calculation, we have, for every place $v$ and every $\alpha,\beta\in F_v$ that
%
%\begin{equation}\label{interchange}
%\Phi\left(\mtrx{uy&&\\&u&\\&&1}\mtrx{1&\alpha&\beta\\&1&\\&&1}\right)=\check{\Phi}\left(\mtrx{y&&\\&u^{-1}&\\&&1}\mtrx{1&\beta&\alpha\\&1&\\&&1}\right).
%\end{equation}
%
%Notice the switch in the positions of $\alpha$ and $\beta$. 
%
%It follows from \eqref{interchange} and the definitions of $U_{\mf{q}}$ and $V_{\mf{q}}$ that for every prime ideal $\mf{p}$ and every integer $n$, one has
%
%$$
%\begin{cases}
%U_{\mf{p}^n}\Phi\left(\mtrx{uy&&\\&u&\\&&1}\right)=V_{\mf{p}^n}\check{\Phi}\left(\mtrx{y&&\\&u^{-1}&\\&&1}\right),\\
%V_{\mf{p}^n}\Phi\left(\mtrx{uy&&\\&u&\\&&1}\right)=U_{\mf{p}^n}\check{\Phi}\left(\mtrx{y&&\\&u^{-1}&\\&&1}\right)
%\end{cases}
%$$
%
%Now, since $U_{\mf{p}^n},V_{\mf{p}^n}$ commute with $U_{\mf{p}'^{n'}},V_{\mf{p}'^{n'}}$ for $\mf{p}\neq \mf{p}'$. Hence, for $(\mf{q},\mf{l})=1$,
%
%$$
%U_{\mf{l}}V_{\mf{q}}\Phi\left(\mtrx{uy&&\\&u&\\&&1}\right)=U_{\mf{q}}V_{\mf{l}}\check{\Phi}\left(\mtrx{y&&\\&u^{-1}&\\&&1}\right).
%$$

Applying this to \eqref{LwAs=intint} and making the change of variables $(u,y)= (u'^{-1},u'y')$ gives the result.

\end{proof}
\section{Spectral expansion of the period}\label{period-expansion}

In this section we will give a spectral decomposition of the period $I(w,\mc{A}_s\Phi)$.
%
%{\color{blue}
%\begin{prop}\label{spectral Whittaker} Let $F\in \mc{C}^{\infty}\left(Z_2(\Aa_F)\G_2(F)\bs \G_2(\Aa_F)/K_0(\mathfrak{q})\right)$ and of rapid decay, then
%
%\begin{multline*}
%F(g) =\int_{N_2(F)\bs N_2(\Aa)}F(ng)\d n + \sum_{\substack{\pi\text{ cusp}\\\operatorname{cond}(\pi)\mid \mf{q}}}\sum_{\phi\in\mc{B}\left(\pi, \mf{q}\right)}\IP{ F, \phi}\sum_{y\in F^{\times}}W_{\phi}(a(y)g)\\
Let $\Pi$ be an automorphic cuspidal representation for $\GL(3)$ over $F$ and let $\Phi=\otimes_v\Phi_v\in \Pi$ be a cusp form. Let $S$ be a finite set of places containing all archimedean places and all the places for which $\Phi$ is not normalized spherical. 
Since $\Phi$ is of rapid decay, then the same holds for $\mathcal{A}_s\Phi$. More precisely this follows by combining the rapid decay of Whittaker functions with the action of the Weyl group of $\GL(3)$. We can thus spectrally decompose it as in Proposition \ref{spectral theorem}:
\begin{multline*}
  \mc{A}_s\Phi(h) = \sum_{\pi\in C(S)}\sum_{\phi\in\mc{B}_c\left(\pi\right)}\IP{ \mc{A}_s\Phi, \phi}\phi(h) +\operatorname{vol}(\mathbf{X})^{-1}\sum_{\substack{\omega\in\Xi(S)\\\omega^2=1}}\IP{\mc{A}_s\Phi,\omega\circ \det}\omega(\det g)\\
+\frac{1}{4\pi}\sum_{\omega\in\Xi(S)}\int_{-\infty}^{\infty}\sum_{\substack{\phi\in\mc{B}_e\left(\pi(\omega,it)\right)}}\IP{\mc{A}_s\Phi,\phi}\phi(h)\d t.
\end{multline*}

%As a first simplification we notice that the contribution from the one-dimensional representations vanish. Indeed
%
%$$
%\IP{\mc{A}_s\Phi,\phi_{\omega}}=\int_{\G_2(F)\bs G_2(\Aa_F)}\Phi\left(\mtrx{h&\\&1}\right)\omega(\det h)|\det h|^{s-\frac12}\d h
%$$
%
%To see that the integral vanishes, we simply make the change of variables $h\leftarrow nh$ for $n\in N(\Aa_F)$ and integrate over the compact $N(F)\bs N(\Aa_F)$. Since $\det nh=\det h$ and $\Phi$ is cuspidal, the claim is verified.

By integrating both sides of the above expression against an additive character and over the compact set $N_2(F)\bs N_2(\Aa)$, we get the following relation for Whittaker functions:
$$
W_{\mc{A}_s\Phi}(h)= \sum_{\pi\in C(S)}\sum_{\phi\in\mc{B}_c\left(\pi\right)}\IP{ \mc{A}_s\Phi, \phi}W_{\phi}(h)+\frac{1}{4\pi}\sum_{\pi\in \Xi(S)}\int_{-\infty}^{\infty}\sum_{\substack{\phi\in\mc{B}_e\left(\pi(\omega,it)\right)}}\IP{\mc{A}_s\Phi,\phi}W_{\phi}(h)\d t.
$$
Notice that since the one-dimensional representations are not generic, they do not contribute to the above expression. Now, because of rapid decay of the Whittaker functions $W_{\phi}$ as $|y|\rightarrow +\infty$, then if we take $\Ree(w)$ sufficiently large, we get

\begin{multline}\label{no-constant-term}
\Psi(w,\mc{A}_s\Phi) = \sum_{\pi\in C(S)}\sum_{\phi\in\mc{B}_c\left(\pi\right)}\IP{ \mc{A}_s\Phi, \phi}\Psi(w,W_{\phi})\\
+\frac{1}{4\pi}\sum_{\omega\in\Xi(S)}\int_{-\infty}^{\infty}\sum_{\substack{\phi\in\mc{B}_e\left(\pi(\omega,it)\right)}}\IP{\mc{A}_s\Phi,\phi}\Psi(w,W_{\phi})\,\d t.
\end{multline}
By using the Fourier decomposition of $\mc{A}_s\Phi$, we see that
$$
\int_{\Aa^{\times}}W_{\mc{A}_s\Phi}\left(a(y)\right)|y|^{w-\frac12}\d^{\times}y=I(w,\mc{A}_s\Phi) - I(w,(\mathcal{A}_s\Phi)_0),
$$
where for any $\phi$ on $\mathcal{C}^{\infty}(\mathbf{X})$, $\phi_0$ is given by 

$$
\phi_0(h):=\int_{F\backslash \mathbb{A}}\phi(n(x)h)\d x.
$$
The next step  is to realize the terms $\IP{\mc{A}_s\Phi,\phi}$ and $\Psi(w,W_{\phi})$ as a product of local integrals. First, it follows from Proposition \ref{RS} that if $\phi=\otimes_v\phi_v$ is decomposable,
$$
\Psi\left(w,W_{\phi}\right)=\prod_v\Psi_v(w,W_{\phi_v}).
$$
Moreover, from the definition of $\mc{A}_s\Phi$, we deduce, after changing variables that
\begin{equation*}
\IP{\mc{A}_s\Phi,\phi}=I(s,\Phi,\overline{\phi}).
\end{equation*}
Since $\Phi$ is a cusp form on $\GL(3)$, it follows from \eqref{I=W-cusp} and Proposition \ref{RS}, that for $\Ree(s)$ sufficiently large and factorable $\phi$,

$$
I(s,\Phi,\overline{\phi})=\Psi(s,W_{\Phi},W'_{\overline{\phi}})=\prod_{v}\Psi_v(s,W_{\Phi_v},W'_{\overline{\phi_v}}),
$$
where $\Psi_v(s,W,W')$ is given by \eqref{Psiv}. As a consequence, we have

\begin{multline}\label{spectral-final}
  I\left(w,W_{\mc{A}_s\Phi}\right)= I\left(w,(\mc{A}_s\Phi)_0\right)+\sum_{\pi\in C(S)}\sum_{\phi\in\mc{B}_c\left(\pi\right)} \prod_{v}\Psi_v\left(s,W_{\Phi_v},W'_{\overline{\phi}_v}\right)\prod_v\Psi_v(s,W_{\phi_v})\\
+\frac{1}{4\pi}\sum_{\omega\in \Xi(S)}\int_{-\infty}^{\infty}\sum_{\substack{\phi\in\mc{B}_e\left(\pi(\omega,it)\right)}}\prod_{v}\Psi_v\left(s,W_{\Phi_v},W'_{\overline{\phi}_v}\right)\prod_v\Psi_v(s,W_{\phi_v})\,\d t.
\end{multline}

For each generic automorphic representation $\pi$ we will now construct an orthonormal basis for $V_{\pi}$ which is formed exclusively by factorable vectors: We start by choosing for each place $v$, an orthogonal basis $\mc{B}^W(\pi_v)$ of the space $\mc{W}(\pi_v,\psi_v)$. Consider now the elements $\phi=\otimes_v\phi_v$ such that for every finite $v$, $W_{\phi_v}$ lies in $\mc{B}^W(\pi_v)$ and for for all but finitely many $v$, $W_{\phi_v}$ is normalized spherical. This provides us with an \textit{orthogonal} basis for $V_{\pi}$. In order to get an \textit{orthonormal} basis we multiply these vectors by the correcting factors coming from \eqref{relation between inner products}. Applying these steps to \eqref{spectral-final} leads to the following (the slightly awkward normalization is justified by the last part of Proposition \ref{RS}):

\begin{prop}\label{spectral-terms-almost}
Let $\Pi$ be a cuspidal automorphic representation and let $\Phi=\otimes_v\Phi_v\in \Pi$ be a cusp form. Then, for complex numbers $s$ and $w$ with sufficienlty large real parts, we have

\begin{equation*}
  2d_F^{\frac72-3s-w}I\left(w,\mc{A}_s\Phi\right)= \mathcal{M}_{s,w}(\Phi) + \mathcal{D}_{s,w}(\Phi),
\end{equation*}
where 
\begin{equation}\label{H-def}
  H(\pi)=\prod_vH_v(\pi_v),\;\;H_v(\pi_v):=p_v^{d_v(3-3s-w)}\!\!\!\!\sum_{W\in\mc{B}^W(\pi_v)}\frac{\Psi_v(s,W_{\Phi_v},\overline{W})\Psi_v(w,W)}{L(s,\Pi_v\times \pi_v)L(w,\pi_v)},
\end{equation}
$\mathcal{M}_{s,w}(\Phi)$ is as in \eqref{M=C+E} and
\begin{multline}\label{D-def}
  \mathcal{D}_{s,w}(\Phi):=
  %2d_F^{\frac72-3s-w}I(w,(\mathcal{A}_s\Phi)_0) =
  2d_F^{\frac72-3s-w}\int_{F^{\times}\backslash \mathbb{A}^{\times}}\int_{F\backslash \mathbb{A}}\int_{F^{\times}\backslash \mathbb{A}^{\times}} \Phi\begin{pmatrix} z(u)n(x)a(y)&\\&1 \end{pmatrix}\\
\times
|u|^{2s-1}|y|^{s+w-1}\d^{\times}u\,\d x\,\d^{\times}y.
\end{multline}
\end{prop}

We will refer to the function $H$ given by \eqref{H-def} where $\Phi=\otimes \Phi_v\in\Pi$ as the $(s,w)-$weight function of kernel $\Phi$. If $s$ and $w$ and $\Phi$ are clear from context, we shall refer to it simply as the weight function of kernel $\Phi$. 

Finally, given $s,w\in\Cc$, if $H$ is the $(s,w)-$weight function with kernel $\Phi$, we let $\widecheck{H}$ be the $(s',w')-$weight function associated to $\widecheck{\Phi}$, where $s'$ and $w'$ are as in \eqref{s'w'} and $\widecheck{\Phi}$ is as in \eqref{check-Phi}. In other words,

\begin{equation}\label{H-check}
\widecheck{H}(\pi)=\prod_v\widecheck{H}_v(\pi_v),\;\; \widecheck{H}_v(\pi):=p_v^{d_v(3-3s'-w')}\sum_{W\in\mc{B}^W(\pi_v)}\frac{\Psi_v(s',W_{\widecheck{\Phi}_v},\overline{W})\Psi_v(w',W)}{L(s',\Pi_v\times \pi_v)L(w',\pi_v)}.
\end{equation}

\section{Local computations}\label{local-comp}

Let $\Pi$ be an unramified cuspidal automorphic representation of $\PGL(3)$ over $F$. For all $v$, we let $\Phi_v^{0}$ correspond to the normalized spherical vector in the Whittaker model, that is $W_{\Phi^0_v}=W_{\Pi_v}$. Let $\mf{q}$ and $\mf{l}$ be two coprime integral ideals of $F$. Finally, let $\Phi^{\mf{q},\mf{l}}=\otimes_v\Phi^{\mf{q},\mf{l}}_v$ where, for all $v\nmid \mf{q}\mf{l}$, we put $\Phi^{\mf{q},\mf{l}}_v=\Phi^{0}_v$, for $v\mid \mf{q}$,

\begin{equation}\label{Phi-q}
  \Phi_v^{\mf{q},\mf{l}}(g):=\frac{1}{p_v^n}\sum_{\beta\,\in\, \mf{m}_v^{-n}/\mf{o}_v}\Phi^0_v\left(g\mtrx{1&&\beta\\&1&\\&&1}\right),
\end{equation}
where $n=v(\mf{q})$, and, for $v\mid \mf{l}$,

$$
\Phi_v^{\mf{q},\mf{l}}(g):=\frac{1}{p_v^m}\sum_{\beta\,\in\, \mf{m}_v^{-m}/\mf{o}_v}\Phi^0_v\left(g\mtrx{1&\beta &\\&1&\\&&1}\right),
$$
with $m=v(\mf{l})$. 

%In this section we compute $H_v$ for all values of $v$. By construction of $\Phi^{\mf{q},\mf{l}}$, we make take $S=\{v\mid\mf{q}\}\cup\{v\mid\mf{l}\}\cup\{v\mid\infty\}$. If $v\not\in S$, we know that $\pi_v$ is unramified and $\Phi_v^{\mf{q},\mf{l}}=\Phi_v^0$ is spherical. We may thus restrict the sum defining $H_v(\pi_v)$ to a basis of space of vectors that are right invariant by $K_v$. This space is one dimensional and the last part of Proposition \ref{RS} gives $H_v(\pi_v)=1$ in that case. We consider the cases where $v\mid \mf{l}$, $v\mid \mf{q}$ and $v\mid \infty$, in that order. %Since the calculation made in the rest of this section are of local nature we will sometimes avoiding writing the subscript $v$ whenever this cannot cause trouble. For example we will write $\mf{m}$, $\mf{o}$, $p$ and $\mf{p}$ instead of $\mf{m}_v$, $\mf{o}_v$, $p_v$ and $\mf{p}_v$.

We will now proceed to the calculation of $H_v$ for $\Phi=\Phi^{\mathfrak{q},\mathfrak{l}}$. First notice that if, for some compact group $K'_v$ of $\GL_2(F_v)$ we have that $\Phi_v$ is invariant on the right by matrices of the shape $ \begin{pmatrix}k&\\&1  \end{pmatrix}$ ,where $k\in K'_v$, then we may restrict the sum over the basis $\mathcal{B}^W(\pi_v)$ to a sum over a basis of the right $K'_v$-invariant vectors. In particular, if $v<\infty$ and $v\nmid \mathfrak{q}\mathfrak{l}$, this basis will have only one element, which can be taken to be normalized spherical. Thus, by Proposition \ref{RS}, we see that $H_v(\pi_v)=1$ in those cases. We divide the remaining cases in three subcategories: $v\mid \mathfrak{l}$, $v\mid \mathfrak{q}$ and $v\mid \infty$ and treat them in that order.

\subsection{Non-archimedean case I: $v\mid\mf{l}$}\label{v-mid-l}

Even though this is not obvious at first glance, we will show that $H_v$ vanishes unless $\pi_v$ is unramified. First, notice that by right $\GL(2)$-invariance of the Whittaker norm, we have that for every orthonomral basis $\mathcal{B}$ of $\mathcal{W}(\pi_v,\psi_v)$, one may construct another one by taking $\mathcal{B}':=\{\pi_v(h)W,W\in\mathcal{B}\}$. Applying this for $h=\mtrx{1&\beta\\&1}$ for $\beta\in \mf{m}_v^{-m}/\mf{o}_v$, changing variables in the $\GL(3)\times\GL(2)$ Rankin-Selberg integral and summing over $\beta$, we deduce that

$$	
H_v(\pi_v)=p_v^{d_v(3-3s-w)}\sum_{W\in \mc{B}^W(\pi_v)}\frac{\Psi_v(s,W_{\Pi_v},\overline{W})\Psi_v(w,W^{(m)})}{L(s,\Pi_v\times\pi_v)L(w,\pi_v)},
$$
where
$$
W^{(m)}(h):=p^{-m}\sum_{\beta\in \mf{m}_v^{-m}/\mf{o}_v}W\left(h\mtrx{1& -\beta\\&1}\right).
$$
Now, since $W_{\Pi_v}$ is spherical, we may restrict the sum over $\mathcal{B}^W(\pi_v)$ to only one term for which $W$ is the normalized spherical vector. Now by Proposition \ref{RS}, $\Psi_v(s,W_{\Pi_v},\overline{W_{\pi_v}})=p_v^{d_v(3s-5/2)}L(s,\Pi_v\times \pi_v)$ and

\begin{align}\label{local-d-div-l}\Psi_v(w,W^{(m)}_{\pi_v})&=\int_{F_v^{\times}}\delta_{v(y)\geq m-d}W_{\pi_v}\mtrx{y&\\&1}|y|^{w-\frac12}\d^{\times} y,\notag\\
&=p_v^{d_v(w-\frac12)}\sum_{\mu\geq m}\frac{\lambda_{\pi_v}(\mu)}{p_v^{\mu w},}=p_v^{-mw}\left(\lambda_{\pi_v}(m)-\frac{\lambda_{\pi_v}(m-1)}{p_v^w}\right)p_v^{d_v(w-\frac12)}L(w,\pi_v).
\end{align}
Hence we have that
\begin{equation}\label{Hv-for-l}
  H_v(\pi_v)=p_v^{-mw}\left(\lambda_{\pi_v}(m)-\frac{\lambda_{\pi_v}(m-1)}{p_v^{w}}\right).
\end{equation}

\subsection{Non-archimedean case II: $v\mid \mf{q}$}\label{v-mid-q}

We will show that $H_v(\pi_v)$ vanishes unless $c(\pi_v)\leq n$, that $H_v(\pi_v)\ll_{\epsilon} p_v^{n(\theta-1+\epsilon)}$ and if $c(\pi_v)=n$, $H_v(\pi_v)=\varphi(p_v^n)p_v^{-2n}$.

We first notice that by a result of Casselman \cite{Casselman1973some}, if we let $W_0=W_{\pi_v}$ be the newvector and for each $j\geq 0$, we let

\begin{equation}\label{Wj}
W_j:=\pi_v\mtrx{1&\\&\varpi_v^j}W_0.
\end{equation}
Then, for each $j\geq 0$, $\left\{W_0,\,W_1,\ldots W_j\right\}$ is a basis for the $K_v[n_0+j]$-invariant vectors in $\mc{W}(\pi_v,\psi_v)$, where $n_0=c(\pi_v)$. We now construct an \textit{orthonormal} basis by employing the process of Gram-Schmidt. This is the local counterpart of the method in \cite{BM2015second}.

Let $\lambda_{\pi_v}=\lambda_{\pi_v}(1)$ be as in \S \ref{newforms}, $\delta_{\pi_v}=\delta_{n_0=0}$ and take $\alpha_{\pi_v}:=\frac{\lambda_{\pi_v}}{\sqrt{p_v}(1+\frac {\delta_{\pi_v}}{p_v})}$. We put

$$
\xi_{\pi_v}(0,0)=1,\;\;\;\xi_{\pi_v}(1,1)=\frac1{\sqrt{1-\alpha_{\pi_v}^2}},\;\;\;\xi_{\pi_v}(1,0)=-\alpha_{\pi_v}p_v^{1/2}\xi_{\pi_v}(1,1),
$$
and
$$
\xi_{\pi_v}(j,j)=\frac1{\sqrt{1-\alpha_{\pi_v}^2}\sqrt{1-\frac{\delta_{\pi_v}}{p_v^2}}},\;\;\;\xi_{\pi_v}(j,j-1)=-\lambda_{\pi_v}\xi_{\pi_v}(j,j),\;\;\;\xi_{\pi_v}(j,j-2)=\delta_{\pi_v}\xi_{\pi_v}(j,j).
$$
and $\xi_{\pi_v}(j,k)=0$ for $k\leq j-2$. If one assumes any non-trivial bound towards ramanujan $\lambda_{\pi_v}\ll p_v^{\vartheta}$, with $\vartheta< 1/2$, one has that $|\alpha_{\pi_v}|$ is uniformly bounded by some constant $C_{\vartheta} <1$ and therefore
\begin{equation}\label{xi-ramanujan}
\xi_{\pi_v}(j,k)\ll p_v^{j\epsilon}p_v^{(j-k)\vartheta}.
\end{equation}
More importantly, for $j\geq 0$, $\{\widetilde{W_0},\widetilde{W_1},\ldots,\widetilde{W_j}\}$ is an orthonormal basis for the space of $K_v[n_0+j]$-vectors in $\mathcal{W}(\pi_v,\psi_v)$, where
\begin{equation}\label{tilde-basis}
  \widetilde{W_j}:=\frac{1}{\IP{W_0,W_0}^{1/2}}\sum_{k=1}^j\xi_{\pi_v}(j,k)p_v^{\frac{k-j}{2}}W_k.
\end{equation}
To see this we first compute $\IP{W_{k_1},W_{k_2}}$, which, by \eqref{Wj} and the definition of $\lambda_{\pi_j}$, equals

$$
p^{-\frac{|k_2-k_1|}{2}}S_{|k_2-k_1|},
$$
where, for $t\geq 0$,
$$
S_t=\frac{\zeta_v(2)}{L_v(1,\pi_v\times \overline{\pi_v})}\sum_{\nu \geq 0}\frac{\lambda_{\pi_v}(\nu)\lambda_{\pi_v}(\nu+t)}{p_v^{\nu}}.
$$
It follows from the Hecke relations for $\lambda_{\pi_v}(\nu)$ that
$$
S_t=\lambda_{\pi_v}S_{t-1}-\delta_{\pi_v}S_{t-2}\text{, for $t\ge 2$ and }S_1=\alpha_{\pi_v} p_v^{1/2}S_0,
$$
from what the claim follows.

By definition, we have

\begin{align*}
W_{\Phi_v}\mtrx{a&b&\\c&d&\\&&1}=\frac{1}{p_v^n}\sum_{\beta\,\in\,\mf{m}_v^{-n}/\mf{o_v}}\psi(\beta c)W_{\Pi_v}\mtrx{a&b&\\c&d&\\&&1}\\
=\delta_{v(c)\geq n-d_v}W_{\Pi_v}\mtrx{a&b&\\c&d&\\&&1}.
\end{align*}

Hence, if we write $h=z(u)n(x)a(y)k$, with $x\in F_v$, $u,\, y\in F_v^{\times}$ and $k=(k_{ij})\in K_v$, then $W_{\Phi_v}$ vanishes unless $v(uk_{21})\geq n-d_v$. Letting $d_1:=\min(n,v(u)+d_v)$ and $d_2:=n-d_1$, we see that this is equivalent to $k$ belonging to  $K_v[d_2]$. This allows us to write
\begin{equation}\label{using-choice-Phi}
  \Psi_v(s,W_{\Phi_v},\overline{W})=p_v^{d_v(2s-1)}\sum_{d_1+d_2=n}\sum_{\min(\nu_1,n)=d_1}p_v^{-2\nu_1(s-\frac12)}\Psi_{\nu_1,d_2}(W),
\end{equation}
where

$$
\Psi_{\nu_1,d_2}(W)=\int_{F_v^{\times}}\int_{K_v[d_2]}W_{\Pi_v}\left(z(\varpi_v^{\nu_1-d_v})a(y)\right)\overline{W}(a(y)k)|y|^{s-\frac32}\d^{\times}y\,\d k.
$$

Now, if $W=\widetilde{W_j}$ is an element of our basis, given by \eqref{tilde-basis}, then it follows that

\begin{equation}\label{integration-over-compact}
\int_{K_v[f]}\widetilde{W_j}(hk)\d k=
\begin{cases}
\operatorname{vol}(K_v[f])\widetilde{W_j}(h),\text{ if }j+n_0\leq f,\\
0,\text{ otherwise}.
\end{cases}
\end{equation}
We reason as follows: On the one hand, for every $j$, $\widetilde{W_j}$ is $K_v[n_0+j]$-invariant and is orthogonal to $\mc{W}(\pi_v,\psi_v)^{K_v[n_0+j-1]}$. On the other, the operator
$$
W\mapsto \frac{1}{\operatorname{vol}(K_v[f])}\int_{K_v[f]}\pi_v(k)W\d k
$$
is the orthogonal projection into the space of $K_v[f]$-invariant vectors.

Applying \eqref{using-choice-Phi} and \eqref{integration-over-compact} to the definition of $H_v(\pi_v)$ and changing order of summation, we are led to

\begin{align}\label{last-before-sum}
  H_v(\pi_v)=\frac{p_v^{d_v(2-s-w)}}{L(s,\Pi_v\times\pi_v)L(w,\pi_v)}\sum_{d_1+d_2=n}\operatorname{vol}(K_v[d_2])\sum_{j\leq d_2-n_0}\sum_{\min(\nu_1,n)=d_1}p_v^{-2\nu_1(s-1/2)}\notag\\
\times\int_{F_v^{\times}}W_{\Pi_v}\mtrx{z(\varpi_v^{\nu_1-d_v})a(y)&\\&1}\overline{\widetilde{W}_j}(a(y))|y|^{s-3/2}\d^{\times}y\Psi_v(w,\widetilde{W_j}),
\end{align}
By letting $\lambda_{\pi_v,j}(\mu)=\widetilde{W_j}(a(\varpi_v^{\mu-d_v}))p^{\frac{\mu}{2}}$ and using \eqref{only-small}, we see that
\begin{multline}\label{integral-fixed-nu1}
\int_{F_v^{\times}}W_{\Pi_v}\mtrx{z(\varpi_v^{\nu_1-d_v})a(y)&\\&1}\overline{\widetilde{W}_j}(a(y))|y|^{s-3/2}\d^{\times}y\\
=p_v^{d_v(s-\frac32)}\sum_{\nu_2\geq 0}\lambda_{\pi_v}(\nu_2,\nu_1)\overline{\lambda_{\pi_v,j}(\nu_2)}p_v^{-\nu_1}p_v^{-\nu_2s}
\end{multline}
and also,
\begin{equation}\label{using-def-lambdaj}
\Psi_v(w,\widetilde{W_j})=p_v^{d_v(w-\frac12)}\sum_{\mu\geq 0}\lambda_{\pi_v,j}(\mu)p_v^{-\mu w}.
\end{equation}
Inserting \eqref{integral-fixed-nu1} and \eqref{using-def-lambdaj} in \eqref{last-before-sum}, we deduce that

\begin{multline*}
  H_v(\pi_v)=\frac{1}{L(s,\Pi_v\times \pi_v)L(w,\pi_v)}\sum_{d_1+d_2=n}\operatorname{vol}(K_v[d_2])\sum_{j\leq d_2-n_0}\\
\sum_{\min(\nu_1,n)=d_1}\sum_{\nu_2\geq 0}\sum_{\mu\geq 0}\frac{\lambda_{\Pi_v}(\nu_2,\nu_1)\overline{\lambda_{\pi_v,j}(\nu_2)}\lambda_{\pi_v,j}(\mu)}{p_v^{(2\nu_1+\nu_2)s}p_v^{\mu w}}.
\end{multline*}

Combining \eqref{Wj}, \eqref{tilde-basis}, %and the identity
we get
$$
\lambda_{\pi_v,j}(\nu)=\IP{W_0,W_0}^{-1/2}\sum_{k=1}^{\min(j,k)}\xi_{\pi_v}(j,k)p_v^{k-\frac j2}\lambda_{\pi_v}(\nu-k)\delta_{\nu\ge k}.
$$

As a consequence, we deduce, after changing variables that

\begin{align*}
  H_v(\pi_v)=\frac{  \IP{W_0,W_0}^{-1}}{L(s,\Pi_v\times \pi_v)L(w,\pi_v)}\sum_{d_1+d_2=n}\operatorname{vol}(K_v[d_2])\sum_{j\leq d_2-n_0}p_v^{-j}\sum_{k_1,k_2\leq j}\xi_{\pi_v}(j,k_1)\xi_{\pi_v}(j,k_2)\\
\times p_v^{k_1(1-w)}p_v^{k_2(1-s)}p_v^{-2d_1s}\sum_{\min(\nu_1,d_2)=0}\sum_{\nu_2\geq 0}
\frac{\lambda_{\Pi_v}(\nu_2+k_2,\nu_1+d_1)\overline{\lambda_{\pi_v}(\nu_2)}}{p_v^{(2\nu_1+\nu_2)s}}\sum_{\mu\geq 0}\frac{\lambda_{\pi_v}(\mu)}{p_v^{\mu w}}.
\end{align*}

We recognize the last sum as $L(w,\pi_v)$, so that 

\begin{multline}\label{last-of-long-sums}
  H_v(\pi_v)=\frac{\IP{W_0,W_0}^{-1}}{L(s,\Pi_v\times \pi_v)}\sum_{d_1+d_2=n}\operatorname{vol}(K_v[d_2])\sum_{j\leq d_2-n_0}p_v^{-j}\sum_{k_1,k_2\leq j}\xi_{\pi_v}(j,k_1)\xi_{\pi_v}(j,k_2)\\
\times p_v^{k_1(1-w)}p_v^{k_2(1-s)}p_v^{-2d_1s}\sum_{\min(\nu_1,d_2)=0}\sum_{\nu_2\geq 0}
\frac{\lambda_{\Pi_v}(\nu_2+k_2,\nu_1+d_1)\overline{\lambda_{\pi_v}(\nu_2)}}{p_v^{(2\nu_1+\nu_2)s}}.
\end{multline}

We are now ready to prove the following

\begin{prop}\label{prop-local-comp}
Let $\Phi_v=\Phi_v^{\mf{q},\mf{l}}$ be as in \eqref{Phi-q} and let  $n=v(\mf{q})$. Then for every $\epsilon>0$ there exists $\delta>0$ such that
\begin{enumerate}[(i)]
  \item $H_v(\pi_v)$ vanishes if $c(\pi_v) >n$.
  \item $H_v(\pi_v)=\varphi(p_v^n)p_v^{-2n}$ if $c(\pi_v) = n$.
\item $H_v(\pi_v)\ll_{\epsilon} p_v^{n(\theta-1+\epsilon)}$ in general for $\Ree(s),\Ree(w)>\frac12-\delta$ and $\delta>0$ sufficiently small.
\end{enumerate}
\end{prop}
 
The first assertion follows by observing that if $n_0=c(\pi_v)>n$ then the sum over $j$ in \eqref{last-of-long-sums} will vanish independently of the value of $d_2$.

The second one holds because if $n_0=n$, we automatically have $d_2=n$ and $d_1=j=k_1=k_2=0$ and moreover, we have

$$
\IP{W_0,W_0}=\frac{\zeta_v(2)}{L_v(1,\pi_v\times\overline{\pi_v})}\sum_{n\ge 1}\frac{|\lambda_{\pi_v}(n)|^2}{p_v^{n}}=
\begin{cases}
  1,&\text{ if }n_0=0,\\
  \zeta_v(2),&\text{ otherwise.}
\end{cases}
$$
Hence,

$$
H_v(\pi_v)=\frac{\varphi(p_v^n)p_v^{-2n}}{L(s,\Pi_v\times \pi_v)}\sum_{\nu_2\geq 0}\lambda_{\Pi_v}(\nu_2,0)\lambda_{\pi_v}(\nu_2),
$$
and we conclude by \eqref{L-factor}.

Finally, in order to show \textit{(iii)}, we apply the estimate in \eqref{xi-ramanujan} and the bounds
$$
\lambda_{\Pi_v}(\nu_1,\nu_2)\ll p_v^{(\nu_1+\nu_2)\theta},\;\;\;\lambda_{\pi_v}(\nu)\ll p_v^{\nu\vartheta}
$$
to \eqref{last-of-long-sums}, which gives

$$
H_v(\pi_v)\ll_{\epsilon} p_v^{n(-1+\epsilon)}\sum_{d_1+d_2=n}p^{2(d_1+j)\delta}\sum_{j=0}^{d_2-n_0}\sum_{k_1,k_2=0}^{j}p_v^{(k_1+k_2 -2j)(\frac12-\vartheta)}p_v^{(d_1+k_2)\theta}
$$
for $\Re(s)>\frac{\theta}{2},\theta+\vartheta$ and it follows from the results in \cite{LRS1999ramanujan} and the Kim-Sarnak bound \cite[Appendix 2]{Kim2003functoriality} that one has $\theta+\vartheta<1/2$. We conclude by taking $\delta$ sufficiently small.

\subsection{Local computations, the archimedean case}\label{local-Archimedean}

The analysis of the archimedean weight functions is of a somewhat different nature from the non-archimedean case. For those places, we make the simplest choice imaginable. Namely we impose that $\Pi_v$ is unramified and $\Phi_v$ is normalized spherical for every archimedean place $v$. As a consequence it easily follows that $H_v(\pi_v)$ vanishes unless $\pi_v$ is itself unramified, in which case we may choose a basis of ${B}^W(\pi_v)$ such that each term corresponds to a different $K$-type, then there will be at most one element $W$ of $\mathcal{B}^W(\pi_v)$ for which the period $\Psi_v(s,W_{\Phi_v},W)$ is non-vanishing and it must be a spherical vector for $\pi_v$.  Moreover, it follows from Stade's formula \cite[Theorem 3.4]{Stade2001Mellin} that
$$
\Psi_v(s,W_{\Phi_v},\overline{W})=L_v(s,\Pi_v\times \pi_v),\;\Psi_v(w,W)=L_v(w,\pi_v).
$$
where $W=W_{\pi_v}\in\mc{W}(\pi_v,\psi_v)$ is spherical and such that $\vartheta_v(W,W)=1$. In particular, the following holds:

\begin{prop}\label{Archimedean calculation}
Let $v$ be an archimedean place of $F$. Let $\Pi_v$ be an irreducible admissible generic representation for $\GL_3(F_v)$, then there exists a vector $\Phi_v\in\Pi_v$ such that for every irreducible admissible generic representation for $\GL_2(F_v)$, we have

$$
H(\pi_v)=
\begin{cases}
1,\text{ if }\pi_v\text{ is unramified},\\
0,\text{ otherwise}.
\end{cases}
$$
\end{prop}

\subsection{Meromorphic continuation with respect to the spectral parameter}\label{merom-wrt-spec-param}

Let $\Phi^{\mathfrak{q},\mathfrak{l}}$ be as in \eqref{Phi-q} and let $H$ be the $(s,w)-$weight function of kernel $\Phi^{\mathfrak{q},\mathfrak{l}}$. Our goal in this section is to find that for any unitary character $\omega$ of $F_v^{\times}$ there is a domain of $\mathbb{C}^3$ on which the function
$$
(s,w,t)\mapsto H_v(\pi_v)
$$
is meromorphic with respect to all three variables with only finitely many polar divisors, where $\pi_v=\pi_v(\omega_v,it)$ (\textit{cf.} \S \ref{Ind-Eis}).

From our computations so far, we know that $H_v(\pi_v)=1$ unless $v\mid \mathfrak{l}$ or $v\mid \mathfrak{q}$. Moreover, in the first of these cases, we saw that

$$
H_v(\pi_v)=p_v^{-mw}\left(\lambda_{\pi_v}(m)-\frac{\lambda_{\pi_v}(m-1)}{p_v^w}\right),\;\;m=v(\mf{l}),
$$
which is clearly an entire function with respect to $s$, $w$ and $t$, since it is a combination of terms of the shape $p_v^{\alpha w +\beta t}$, $\alpha,\beta\in \Cc$.

We are now left with the case where $v\mid \mathfrak{q}$. It follows from the proof of Proposition \ref{prop-local-comp} that given $\eta<\frac12$, there exists $\delta>0$ such that the right-hand side of \eqref{last-of-long-sums} converges in the region

$$
|\Ima(t)|<\eta,\;\;\Ree(s),\Ree(w)>\frac12-\delta,
$$
and defines in it a holomorphic function in the variables $s$ and $w$. We observe that $H_v(\pi_v)$ is a linear combination of terms of the shape

%{
%  \color{blue} $q=p_v$, $cf=p^{\nu_2}$ $M=p^{k_2}$, $d=p^{d_1}$, $n=p^{\nu_1}$ $g_1=1$ and $g_2=p^{d_2}$
%}

$$
L_{\omega,k_2,d_1,d_2}(s,t):=\frac{1}{L(s,\Pi_v\times \pi_v)}\sum_{\min(\nu_1,d_2)=0}\sum_{\nu_2\geq 0}
\frac{\lambda_{\Pi_v}(\nu_2+k_2,\nu_1+d_1)\overline{\lambda_{\pi_v}(\nu_2)}}{p_v^{(2\nu_1+\nu_2)s}},
$$
with coefficients given by meromorphic functions in the variables $s$, $w$ and $t$. The only possible polar divisors occur for $t$ satisfying $\omega(\varpi_v)^2p_v^{2it}
=p_v^{\pm 1}$, due to the term $(1-\alpha^2_{\pi_v})^{-1}$ appearing as a factor of $\xi_{\pi_v}(j,k_1)\xi_{\pi_v}(j,k_2)$. Moreover, it follows from \cite[Lemma 14]{BlKh2019reciprocity}, applied to the tuple $(M,d,g_1,g_2,q)=(p_v^{k_2},p_v^{d_1},1,p_v^{d_2},p_v^n)$ that for any $\epsilon>0$, there exists $\delta>0$ such that $L_{\omega,k_2,d_1,d_2}(s,t)$ admits a holomorphic continuation to the region 
\begin{equation}\label{Region-Lkdd}
\Ree(s)>\frac14-\delta\;\;\Ree(s)\pm\Ima(t)>-\delta.
\end{equation}
Moreover, using again the Ramanujan bound for $\lambda_{\Pi_v}(\nu_2,\nu_2)$ and recalling that $\pi_v=\pi_v(\omega_v,it)$ so that $\lambda_{\pi_v}(\nu)\ll p_v^{\nu|\Ima(t)|+\epsilon}$,
%$$
%=\frac{\omega_v(\varpi_v)^{n+1}p^{-(n+1)it}-\omega_v(\varpi_v)^{-n-1}p^{(n+1)it}}{\omega_v(\varpi_v)p^{-it}-\omega_v(\varpi_v)^{-1}p^{it}}
%$$
we see that in Region \eqref{Region-Lkdd} we have 
\begin{equation}\label{bound-for-Lkdd}
L_{\omega,k_2,d_1,d_2}(s,w,t)\ll p_v^{(d_1+k_2)(\theta+\epsilon)}.
\end{equation}
As a consequence, $H_v(\pi_v)$ admits meromorphic continuation to \eqref{Region-Lkdd}.
Now, suppose $\omega=\mathbf{1}$ is the trivial character and let

\begin{equation}\label{G=F}
  D_v(s,w):=H_v(\pi_v)_{\big| t=(1-w)/i}. 
\end{equation}
From what we have just seen, $D_v(s,w)=1$ unless $v\mid \mathfrak{l}$ or $v\mid \mathfrak{q}$. In the first case, it is clear that $D_v(s,w)$ is entire with respect to both $s$ and $w$. Moreover, when $\frac12\leq \Ree(s),\,\Ree(w) < 1$, we have $\lambda_{\pi_v}(m)\ll p_v^{m(1-\Ree(w))}$, and thus by \eqref{Hv-for-l}, we see that
$$
D_v(s,w)\ll 1.%_{\epsilon} p_v^{m\epsilon}
$$
Finally, if $v\mid \mathfrak{q}$, then for sufficiently small $\delta>0$, $D_v$ is meromorphic in the region
\begin{equation}\label{Region-G}
\frac12-\delta<\Ree(s),\Ree(w)\leq 1,
\end{equation}
where the only possible polar divisors are at the values of $w$ such that $p_v^{2-2w}=p_v$. We will now show that such poles cannot occur. To see this, let

$$
E_{j,k_2}(w):=\xi_{\pi_v}(j,k_2)\sum_{k_1=0}^j\xi_{\pi_v}(j,k_1)p_v^{k_1(1-w)},
$$
An easy computation shows that

$$
E_{j,k_2}(w)=
\begin{cases}
1,\text{ if }j=0,\\
\frac{t_{j,k_2}(w)}{(1-p_v^{2w-3})},\text{ if }j=1,\\
0,\text{ otherwise},
\end{cases}
$$
where $t_{j,k_2}$ is an entire function. This and the fact that $L_{\omega,k,d_1,d_2}(s,w,(1-w)/i)$ is  holomorphic in \eqref{Region-G} are enough to guarantee that $H_v(\pi_v)$ is holomorphic in the same region. Moreover, we may argue analogously to Proposition \ref{prop-local-comp} $(iii)$, appealing to \eqref{bound-for-Lkdd}, to deduce that for $\frac12\leq \Ree(s),\,\Ree(w)<1$, we have the inequality

$$
D_v(s,w)\ll_{s,w,\epsilon} p_v^{n(-1+\theta+\epsilon)}\sum_{d_1+d_2=n}\sum_{j=0}^{d_2-n_0}p_v^{j(1-2\Ree(w))}+p_v^{j(1-\Ree(s)-\Ree(w))}\ll_{\epsilon}  p_v^{n(-1+\theta+\epsilon)},
$$
We now summarize what we obtained in this subsection as follows:

\begin{prop}\label{meromorphic-t-variable}
  Let $\omega=\otimes_v'\omega_v$ be an unitary character of $F^{\times}\backslash\Aa^{\times}$ and let $H_v$ be  given by \eqref{H-def} with $\Phi_v = \Phi^{\mf{q},\mf{l}}_v$, with $\Phi^{\mf{q},\mf{l}}_v$ given by \eqref{Phi-q}. Then $(s,w,t)\mapsto H_v(\pi_v(\omega_v,it))$ admits meromorphic continuation to the region \eqref{Region-Lkdd} with possible polar divisor of the form $t=t_0$ where $t_0$ is a solution to $\omega(\varpi_v)^2p_v^{2it_0}
=p_v^{\pm 1}$. Moreover, if $D_v$ is given by \eqref{G=F}, then it admits a \textit{holomorphic} continuation to the region \eqref{Region-G} and if $\frac12\leq \Ree(s),\,\Ree(w)<1$, it satisfies
$D_v(s,w)=1$ unless $v\mid \mathfrak{q}\mathfrak{l}$, in which case,
$$
D_v(s,w)\ll_{s,w,\epsilon}
\begin{cases}
 p_v^{m\epsilon},\text{ if }v\mid \mathfrak{l},\\
 p_v^{n(-1+\theta+\epsilon)},\text{ if }v\mid \mathfrak{q}.\\
  
\end{cases}
$$
\end{prop}

\section{The degenerate term}\label{degenerate-sec}

In this section we study the term $\mathcal{D}_{s,w}(\Phi)$ given by \eqref{D-def} and its companion $\mathcal{D}_{s',w'}(\widecheck{\Phi})$. First, by rapid decay of Whittaker functions and the action of the Weyl group of $\GL(3)$, we may see that both converge for any values of $s,w\in \Cc$. This is all that is needed to know with respect to these terms for Theorem \ref{main}.

Let us now turn to their use in Theorem \ref{Spherical}. Here we make the specialization to $\Phi=\Phi^{\mathfrak{q},\mathfrak{l}}$. It turns out that is easier to study first the term $\mathcal{D}_{s',w'}(\widecheck{\Phi})$ so we start with this one and later deduce an analogous result for the other by using their symmetry. First, we recall that 
\begin{equation}\label{int-for-d-check}
  \mathcal{D}_{s',w'}(\widecheck{\Phi})=2d_F^{\frac72-3s'-w'}\int_{F^{\times}\backslash \mathbb{A}^{\times}}\int_{F\backslash \mathbb{A}}\int_{F^{\times}\backslash \mathbb{A}^{\times}} \widecheck{\Phi}\begin{pmatrix} z(u)n(x)a(y)&\\&1 \end{pmatrix} |u|^{2s'-1}|y|^{s'+w'-1}\d^{\times}u\,\d x\,\d^{\times}y.
\end{equation}
We will show that, in the region
  \begin{equation}\label{range-for-J}
\Ree(3s+w)>1,\Ree(s+w),\Ree(2s)>\theta,
\end{equation}
it satisfies $\mathcal{D}_{s',w'}(\widecheck{\Phi})\ll_{s,w,\epsilon} (N\mathfrak{l})^{\theta-\Ree(s)-\Ree(w)+\epsilon}$.

We begin by noticing that, using the definition of $\widecheck{\Phi}$, reversing the change of variables used in the proof of Proposition \ref{abstract-reciprocity} and changing the order of summation, we see that the integral in \eqref{int-for-d-check} equals

$$
\int_{(F^{\times}\backslash \mathbb{A}^{\times})^2} \left(\int_{F\backslash \mathbb{A}}\Phi\left(\begin{pmatrix} 1&&x\\&1&\\&&1 \end{pmatrix} \begin{pmatrix} z(u)a(y)&\\&1 \end{pmatrix}\right) \d x\right)|u|^{2s-1}|y|^{s+w-1}\d^{\times}u\,\d^{\times}y.
$$
By the Whittaker expansion of $\Phi$, the inner integral is

$$
  \int_{F\backslash \mathbb{A}}\sum_{\gamma \in N_2(F)\backslash \GL_2(F)}W_{\Phi}\left(\begin{pmatrix} \gamma&\\&1 \end{pmatrix} \begin{pmatrix} 1&&x\\&1&\\&&1 \end{pmatrix} \begin{pmatrix} z(u)a(y)&\\&1 \end{pmatrix}\right) \d x,
$$
which by elementary manipulations and changing the order of summation and integration, we arrive at
$$
\sum_{\gamma \in N_2(F)\backslash \GL_2(F)}W_{\Phi} \begin{pmatrix} \gamma z(u)a(y)&\\&1 \end{pmatrix}\int_{F\backslash \mathbb{A}}\psi(\gamma_{21}x) \d x, 
$$
where $\gamma_{21}$ is the lower left entry of $\gamma$. Since $\gamma_{21}\in F$, the inner integral vanishes unless $\gamma_{21}=0$, in which case, it equals one. In other words, we may change the sum over $N_2(F)\backslash\GL_2(F)$ into a sum over $N_2(F)\backslash B_2(F)$, which can be parametrized by $Z_2(F)A_2(F)$. Altoghether, this implies that
\begin{multline}\label{eulerian-for-d}
  \mathcal{D}_{s',w'}(\widecheck{\Phi}) = 2d_F^{\frac72-3s'-w'} \int_{(F^{\times}\backslash \mathbb{A}^{\times})^2} \sum_{\gamma\in Z_2(F)A_2(F)}W_{\Phi}\begin{pmatrix} \gamma z(u)a(y)&\\&1 \end{pmatrix}|u|^{2s-1}|y|^{s+w-1}\d^{\times}u\,\d^{\times}y\\
  =2d_F^{\frac72-3s'-w'}\int_{(\mathbb{A}^{\times})^2} W_{\Phi}\begin{pmatrix} z(u)a(y)&\\&1 \end{pmatrix}|u|^{2s-1}|y|^{s+w-1}\d^{\times}u\,\d^{\times}y.
\end{multline}

Suppose that $\Ree(s)$ and $\Ree(w)$ are sufficiently large. We are now in a fairly advantageous position as the integral above can be factored into local ones. These local integrals are

$$
  \mathcal{J}_v = \int_{(F_v^{\times})^2} W_{\Phi_v}\begin{pmatrix} z(u)a(y)&\\&1 \end{pmatrix}|u|^{2s-1}|y|^{s+w-1}\d^{\times}u\,\d^{\times}y.
$$

We notice that for a finite place $v$ for which $\Pi_v$ is unramified, this equals

$$
p_v^{d_v(3s+w-2)}\sum_{\nu_1,\nu_2\geq 0}\frac{\lambda_{\Pi_v}(\nu_1,\nu_2)}{p_v^{\nu_1(s+w)+2\nu_2s}},
$$
whose inner sum we recognize as being the local factor of Bump's double Dirichlet series (see \textit{e.g.} \cite[\S 6.6]{goldfeld-book}). In particular, it follows that for $\Ree(s+w),\Ree(2s)>\theta$ (recall the bound $\lambda_{\Pi_v}(\nu_1,\nu_2)\ll p_v^{(\nu_1+\nu_2)\theta}$), the above equals
\begin{equation}\label{local-for-bump-double}
  \mathcal{J}_v^0:=p_v^{d_v(3s+w-2)}\frac{L(s+w,\Pi_v)L(2s,\overline{\Pi_v})}{\zeta_v(3s+w)} \asymp_{s,w} 1.
\end{equation}
As for the remaining places, we first observe that for $v\mid \mathfrak{q}$, the unipotent averaging has no effect on the values of the Whittaker function at diagonal element. Thus, it follows that the local integral $\mathcal{J}_v$ will also coincide with \eqref{local-for-bump-double}. Furthermore, \eqref{local-for-bump-double} also holds for archimedean $v$. For real places this is done in \cite{bump-book-gl3} and for the complex places this is \cite[Theorem 1]{BF1989Mellin}. Finally, for $v\mid \mathfrak{l}$, we have
$$
W_{\Phi_v}\begin{pmatrix} z(u)a(y) &\\&1 \end{pmatrix}= \delta_{v(y)\geq m-d_v}W_{\Pi_v}\begin{pmatrix} z(u)a(y) &\\&1 \end{pmatrix},
$$
where $m=v(\mathfrak{l})$. Hence, in this case, the local factor is
$$
\mathcal{J}_v=p_v^{d_v(3s+w-2)}\sum_{\nu_1\geq m,\nu_2\geq 0}\frac{\lambda_{\Pi_v}(\nu_1,\nu_2)}{p_v^{\nu_1(s+w)+2\nu_2s}},
$$
which, by using yet again the Ramanujan bound for $\lambda_{\Pi_v}(\nu_1,\nu_2)$, we may see that it converges for $\Ree(s+w),\Ree(2s)>\theta$, where it satisfies

%Moreover, in the same region we have that $L(s+w,\Pi_v)$ and $L(2s,\overline{\Pi_v})$ has no zeros. Thus
%
$$
\mathcal{J}_v\ll_{\epsilon} p_v^{m(\theta-\Ree(s)-\Ree(w)+\epsilon)}.
$$
In particular, if $\mathfrak{l}=1$,

\begin{equation}\label{D-central}
  \mathcal{D}_{\frac12,\frac12}(\widecheck{\Phi})=2d_F^{\frac32}\frac{\Lambda(1,\Pi)\Lambda(1,\overline{\Pi})}{\xi_F(2)}.
\end{equation}

Now, notice that $\mathcal{D}_{s,w}(\Phi)$ is the same as  $\mathcal{D}_{s',w'}(\widecheck{\Phi})$ but with $(\mathfrak{q},\mathfrak{l},s,w)$ replaced by $(\mathfrak{l},\mathfrak{q},s',w')$. This allows us to immediately reuse our efforts in this section to study the latter function as well. We record the results for both these functions in a weaker form in the following proposition.

\begin{prop}\label{degenerate-prop}
  Let $\mathcal{D}_{s,w}(\Phi)$ and $\mathcal{D}_{s,w}(\Phi)$ be as defined in \eqref{D-def} with $\Phi=\Phi^{\mathfrak{q},\mathfrak{l}}$ and $\widecheck{\Phi}$ given by \eqref{check-Phi}. Then they are entire function of $s$ and $w$ and, in the region
$$
\Ree(s),\,\Ree(w),\,\Ree(s'),\,\Ree(w')>\frac14,
$$
they satisfy 
$$
\mathcal{D}_{s,w}(\Phi)\ll_{s,w,\epsilon}(N\mathfrak{q})^{\theta -\Ree(s+w)+\epsilon} 
$$
and
$$
\mathcal{D}_{s',w'}(\widecheck{\Phi})\ll_{s,w,\epsilon}(N\mathfrak{l})^{\theta -\Ree(s+w)+\epsilon}. 
$$
\end{prop}

\section{Analytic continuation of the Eisenstein part}\label{analytic-continuation}

%\begin{multline*}
%\mc{E}_0(\Phi,s,w,\mf{q},\mf{l}))=\frac{\xi^{\ast}_F(1)}{N\mf{q}}\sum_{\substack{\omega\in\widehat{F^{\times}U_{\infty}\bs\Aa^{\times}_{(1)}}\\\operatorname{cond}(\omega)^2\mid \mf{q}}}\int_{-\infty}^{\infty}\frac{\Lambda(s,\Pi\times \pi(\omega,it))\Lambda(w,\pi(\omega,it))}{\Lambda^{\ast}(1,\operatorname{Ad},\pi(\omega,it))}\\
%\times\frac{\widehat{\lambda}_{\pi_{it}}(\mf{l},w)}{(N\mf{l})^w}\mc{Z}^{\mf{q}}_{s,w}(\Pi,\pi_{it})\frac{dt}{2\pi}.
%\end{multline*}

The conclusion of our next Proposition will be subject to the following hypothesis, whose verification when $\Phi=\Phi^{\mf{q},\mf{l}}$ follows from the main results of Section \ref{local-comp}:

\begin{hyp}\label{hypothesis-meromorphic}
There exists $\delta>0$ such that for every idele character $\omega$, the function
$$
(s,w,t)\mapsto H(\pi(\omega,it)) 
$$
is holomorphic in the region 
$$
\Ree(s),\Ree(w)>\frac12-\delta,\;|\Ima(t)|<\delta.
$$
Moreover, $H(\pi(\omega, (1-w))$ admits a holomorphic continuation to the region 
$$
\frac12 -\delta<\Ree(s),\Ree(w)<1.
$$ 
\end{hyp}

We will show that the term $\mathcal{E}_{s,w}(\Phi)$ admits meromorphic continuation for values of $s$ and $w$ with real parts smaller than $1$. The proof follows the same lines as those of \citep[Lemma 16]{BlKh2019reciprocity} and \citep[Lemma 3]{BlKh2019uniform}. 

\begin{prop}\label{analytic-continuation-prop}

Suppose that $\Pi$ is a cuspidal automorphic representation and let $\Phi\in\Pi$ be an automorphic form such that the associated weight function $H$ satisfies Hypothesis \ref{hypothesis-meromorphic} for some $\delta>0$. Let $\mc{E}_{s,w}(\Phi)$ be given by \eqref{E(H)-def}, defined initially for $\Ree(s),\Ree(w)\gg 1$. It admits a meromorphic continuation to $\Ree(s),\Ree(w)\geq \frac12 -\epsilon$ for some $\epsilon>0$ with at most finitely many polar divisors. If $\frac12\leq \Ree(s),\Ree(w)<1$, its analytic continuation is given by $\mc{E}_{s,w}(\Phi)+\mc{R}_{s,w}(\Phi)$, where

\begin{equation}\label{def-R}
  \mc{R}_{s,w}(\Phi)=\sum_{\pm}\underset{\substack{t=\pm (1-w)/i}}{\operatorname{res}}(\pm i)\frac{\Lambda(s+it,\Pi)\Lambda(s-it,\Pi)\xi_F(w+it)\xi_F(w-it)}{\xi_F^{\ast}(1)\xi_F(1+2it)\xi_F(1-2it)}H(\pi(\mathbf{1},it)).
\end{equation}

%In particular, for $s=w=1/2$, $\mf{p}$ a prime ideal and $\Phi=\Phi^{\mf{o}_F,\mf{p}}$, we have
%$$
%\mc{R}_{s,w}(\Phi)= \frac{4\xi^{\ast}_F(1)}{\xi_F(2)}\Lambda(1,\Pi)\Lambda(1,\overline{\Pi}).
%$$
\end{prop}

\begin{proof}

  Let $\delta>0$ to be chosen later. We use non-vanishing of completed Dirichlet $L$-functions $\Lambda(s,\omega)$ at $\Ree(s)=1$ and continuity to define a continuous function $\sigma:\Rr\mapsto (0,\delta)$ so that neither $\Lambda(1-2\sigma-2it,\omega^2)$ nor $H(\pi(\omega,it+\sigma))$ have poles for $0\leq \sigma<\sigma(t)$. 

We start by noticing that we can \textit{analytically} continue $\mc{E}_{s,w}(\Phi)$ to $\Ree(s),\Ree(w)>1$, since in that region, one does not encounter any poles of $\Lambda(w,\pi(\mathbf{1},it))$. Now, suppose that
$$
1<\Ree(s)<1+\sigma(\Ima(s))\text{  and  }1<\Ree(w)<1+\sigma(\Ima(w)).
$$
We shift the contour of the integral defining $\mc{E}_{s,w}(\Phi)$ down to $\Ima{t} = -\sigma(\Ree(t))$. We pick up a pole of $\Lambda(w-it,\omega)$ when $\omega$ is the trivial character and  $w - it = 1$. 

We observe that in view of our choice for $\sigma$, the resulting integral defines a holomorphic function in the region
$$
\begin{cases}
1-\sigma(\Ima(s))<\Ree(s)<1+\sigma(\Ima(s)),\\
1-\sigma(\Ima(t))<\Ree(w)<1+\sigma(\Ima(w)).
\end{cases}
$$

Take now $s$ and $w$ satisfying $1-\sigma(\Ima(s))<\Ree(s)<1$ and $1-\sigma(\Ima(t))<\Ree(w)<1$. We may shift the contour back to the real line  and pick a new pole when $\omega$ is trivial and at $w+it=1$. This proves the desired formula for $1-\sigma(\Ima(s))<\Ree(s)<1$ and $1-\sigma(\Ima(t))<\Ree(w)$ and it follows in general by analytic continuation to all $s,w$ such that $\frac12-\delta<\Ree(s),\Ree(w)<1$ by Proposition \ref{meromorphic-t-variable}.

\end{proof}

\
\section{Conclusion}\label{conclusion}

In this section we put together the results of the last three sections and deduce Theorem \ref{Spherical}. We have seen in Proposition \ref{spectral-terms-almost} that for sufficiently large values of $\Ree(s)$ and $\Ree(w)$, we have the relation
$$
2d_F^{\frac72-3s-w}I(w,\mc{A}_s\Phi)=\mc{M}_{s,w}(\Phi)+\mc{D}_{s,w}(\Phi).
$$

If we assume that $H$ satisfies Hypothesis \ref{hypothesis-meromorphic}, then we may apply Proposition \ref{analytic-continuation-prop} and deduce that for $\frac12-\delta<\Ree(s),\Ree(w)<1$, we have
\begin{equation}\label{I=C+E+R}
  2d_F^{\frac72-3s-w}I(w,\mc{A}_s\Phi)=\mc{M}_{s,w}(\Phi)+\mc{D}_{s,w}(\Phi)+\mc{R}_{s,w}(\Phi).
\end{equation}

Now suppose that $\widecheck{H}$ also satisfies Hypothesis \ref{hypothesis-meromorphic} and furthermore, suppose $\frac12<\Ree(s)\leq \Ree(w)\leq \frac34$. The last assertion implies that
$$
\frac12\leq \Ree(s'),\Ree(w')<1.
$$
Thus, we may deduce that \eqref{I=C+E+R} also holds with $H$, $s$ and $w$ replaced by $\widecheck{H}$, $s'$ and $w'$, respectively. 
The main equality in Theorem \ref{Spherical} is now a direct consequence of Proposition \ref{abstract-reciprocity} and the description of the local weights $H_v(\pi_v)$ from \S \ref{local-comp}. In particular, we show in \S \ref{merom-wrt-spec-param} that the weight function associated to $\Phi^{\mathfrak{q},\mathfrak{l}}$ satisfy Hypothesis \ref{hypothesis-meromorphic}. As for the inequality \eqref{ineq-for-N0}, it follows from Eq. \eqref{def-R}, Proposition \ref{meromorphic-t-variable}, and Proposition \ref{degenerate-prop}.  
 
\subsection{Proof of Corollary \ref{non-vanishing}}
We use Theorem \ref{Spherical} with $s=w=1/2$, $\mf{l}=\mf{o}_F$ and $\mf{q}=\mf{p}$, a prime ideal. We obtain that
$$
\mc{M}(\Phi)=\mc{D}(\widecheck{\Phi})+\mc{R}(\widecheck{\Phi})-\mc{D}(\Phi)-\mc{R}(\Phi)+\mc{M}(\widecheck{\Phi}),
$$
where we dropped the $\frac 12,\frac 12$ from the index for brevity. It follows from Proposition \ref{analytic-continuation-prop} and the fact that $\widehat{\lambda}_{\pi(\mathbf{1},\pm 1/2)}(1/2,\mathfrak{q})=1$ for any ideal $\mathfrak{q}$ that
$$
\mc{R}(\widecheck{\Phi})=2\frac{\Lambda(1,\Pi)\Lambda(0,\Pi)}{\xi_F(2)}.
$$
Furthermore, we have from \eqref{D-central} that
$$
\mc{D}(\widecheck{\Phi})=2d_F^{3/2}\frac{\Lambda(1,\Pi)\Lambda(1,\overline{\Pi})}{\xi_F(2)}=2\frac{\Lambda(1,\Pi)\Lambda(0,\Pi)}{\xi_F(2)}.
$$
Moreover, in view of Proposition \ref{prop-local-comp} and Proposition \ref{degenerate-prop}, we may obtain that
$$
\mc{R}(\Phi),\mc{D}(\Phi)\ll (N\mf{p})^{\theta-1+\epsilon}
$$
and
$$
\mc{M}(\Phi)=\frac{\varphi(N\mathfrak{p})}{(N\mathfrak{p})^2}\sum_{\substack{\pi\text{ cusp}^0\\\operatorname{cond}(\pi)=\mf{p}}}\frac{\Lambda(\frac12,\Pi\times \pi)\Lambda(\frac12,\pi)}{\Lambda(1,\operatorname{Ad},\pi)}+O\left((N\mf{p})^{\theta-1}\mc{M}^{\ast}\right),
$$
where
\begin{multline*}
\mc{M}^{\ast}:=\sum_{\substack{\pi\text{ cusp}^0\\\operatorname{cond}(\pi)=\mf{o}_F}}\frac{|\Lambda(\frac12,\Pi\times \pi)\Lambda(\frac12,\pi)|}{|\Lambda(1,\operatorname{Ad},\pi)|}\\
+\sum_{\substack{\omega\in\widehat{F^{\times}U_{\infty}\bs\Aa^{\times}_{(1)}}\\\operatorname{cond}(\omega)=\mathfrak{o}_F}}\int_{-\infty}^{\infty}\frac{|\Lambda(\frac12,\Pi\times \pi(\omega,it))\Lambda(\frac12,\pi(\omega,it))|}{|\Lambda^{\ast}(1,\operatorname{Ad},\pi(\omega,it))|}\frac{\d t}{2\pi}.
\end{multline*}
Finally, it is easy to see that we also have the bound
$$
\mc{M}(\widecheck{\Phi})\ll (N\mf{p})^{\vartheta-1/2}\mc{M}^{\ast}.
$$

Corollary \ref{non-vanishing} will follow provided that one is able to show $\mc{M}^{\ast}\ll 1$. In other words, we just need to ensure that it converges since it is clearly independent of $\mf{p}$.
To see that, we notice that the finite part of $\frac{\Lambda(\frac12,\Pi\times \pi)\Lambda(\frac12,\pi)}{\Lambda(1,\operatorname{Ad},\pi)}$ is bounded polynomially in terms of the eigenvalues of $\pi_v$ for archimedean $v$ and that by Stirling's formula, we have for archimedean $v$:
$$
\frac{L(\frac12,\Pi_v\times \pi_v)L(\frac12,\pi_v)}{L(1,\operatorname{Ad},\pi_v)}\ll |t_{\pi_v}|^Ce^{-2c_{F_v}\pi|t_{\pi_v}|},
$$
for $\pi_v=\pi_v(\mathbf{1},it_{\pi_v})$, where
$$
c_{F_v}=
\begin{cases}
1,\text{ if }F_v=\Rr,\\
2,\text{ if }F_v=\Cc.
\end{cases}
$$

This implies that the factor $\frac{\Lambda(\frac12,\Pi\times \pi)\Lambda(\frac12,\pi)}{\Lambda(1,\operatorname{Ad},\pi)}$ decay exponentially as the $t_{\pi_v}$ grow and the convergence of $\mc{M}^{\ast}$ follows by appealing to the Weyl law for $\GL(2)$ over number fields (cf. \cite[Theorem 3.2.1]{Palm2012thesis}).

\bibliographystyle{alpha}
\bibliography{references}

\begin{thebibliography}{BHKM20}

\bibitem[AK18]{AK2018level}
Nickolas Andersen and Eren~Mehmet Kıral.
\newblock Level reciprocity in the twisted second moment of rankin–selberg
  {$L$}-functions.
\newblock {\em Mathematika}, 64(3):770–784, 2018.

\bibitem[BF89]{BF1989Mellin}
Daniel Bump and Solomon Friedberg.
\newblock On mellin transforms of unramified whittaker functions on {$\GL(3,
  \Cc)$}.
\newblock {\em Journal of Mathematical Analysis and Applications}, 139(1):205
  -- 216, 1989.

\bibitem[BHKM20]{BHKM2020MOtohashi}
Valentin Blomer, Peter Humphries, Rizwanur Khan, and Micah~B. Milinovich.
\newblock Motohashi’s fourth moment identity for non-archimedean test
  functions and applications.
\newblock {\em Compositio Mathematica}, 156(5):1004--1038, 2020.

\bibitem[BK19a]{BlKh2019reciprocity}
Valentin Blomer and Rizwanur Khan.
\newblock Twisted moments of {$L$}-functions and spectral reciprocity.
\newblock {\em Duke Math. J.}, 168(6):1109--1177, 04 2019.

\bibitem[BK19b]{BlKh2019uniform}
Valentin Blomer and Rizwanur Khan.
\newblock Uniform subconvexity and symmetry breaking reciprocity.
\newblock {\em Journal of Functional Analysis}, 276(7):2315 -- 2358, 2019.

\bibitem[BKL20]{BKL2020test}
Andrew~R. Booker, M.~Krishnamurthy, and Min Lee.
\newblock Test vectors for {R}ankin–{S}elberg {$L$}-functions.
\newblock {\em Journal of Number Theory}, 209:37 -- 48, 2020.

\bibitem[BLM19]{BML2019spectral}
Valentin Blomer, Xiaoqing Li, and Stephen~D. Miller.
\newblock A spectral reciprocity formula and non-vanishing for {$L$}-functions
  on {$\GL(4)\times\GL(2)$}.
\newblock {\em Journal of Number Theory}, 205:1 -- 43, 2019.

\bibitem[BM15]{BM2015second}
Valentin Blomer and Djordje Milićević.
\newblock The second moment of twisted modular {$L$}-functions.
\newblock {\em Geometric and Functional Analysis}, 25(2):453--516, 2015.

\bibitem[Bum84]{bump-book-gl3}
Daniel Bump.
\newblock {\em Automorphic Forms on {$GL(3,\Rr)$}}, volume 1083 of {\em Lecture
  Notes in Mathematics}.
\newblock Springer-Verlag Berlin Heidelberg, 1984.

\bibitem[Cas73]{Casselman1973some}
William Casselman.
\newblock On some results of {A}tkin and {L}ehner.
\newblock {\em Math. Ann.}, 201:301--314, 1973.

\bibitem[CF67]{cassels-frohlich}
John William~Scott Cassels and Albrecht Fr{\"o}hlich.
\newblock {\em Algebraic number theory: proceedings of an instructional
  conference}.
\newblock Academic press, 1967.

\bibitem[Cog07]{Cogdell2007functions}
James~W. Cogdell.
\newblock {$L$}-functions and converse theorems for {${\rm GL}_n$}.
\newblock In {\em Automorphic forms and applications}, volume~12 of {\em
  IAS/Park City Math. Ser.}, pages 97--177. Amer. Math. Soc., Providence, RI,
  2007.

\bibitem[CPS94]{CPS1994converse}
James~W. Cogdell and Ilya~I. Piatetski-Shapiro.
\newblock Converse theorems for {$\mathrm{GL}_n$}.
\newblock {\em Publications Math{\'e}matiques de l'IH{\'E}S}, 79:157--214,
  1994.

\bibitem[GJ79]{GelbartJacquet1979forms}
Stephen Gelbart and Herv{\'e} Jacquet.
\newblock Forms of {$\GL (2)$} from the analytic point of view.
\newblock In {\em Automorphic forms, representations and L-functions (Proc.
  Sympos. Pure Math., Oregon State Univ., Corvallis, Ore., 1977), Part},
  volume~1, pages 213--251, 1979.

\bibitem[Gol06]{goldfeld-book}
Dorian Goldfeld.
\newblock {\em Automorphic Forms and {$L$}-Functions for the Group
  {$\GL(n,\Rr)$}}.
\newblock Cambridge Studies in Advanced Mathematics. Cambridge University
  Press, 2006.

\bibitem[JPSS79]{JPSS1979automorphic}
H.~Jacquet, I.~I. Piatetski-Shapiro, and J.~Shalika.
\newblock Automorphic forms on {${\rm GL}(3)$}. {I \& II}.
\newblock {\em Ann. of Math. (2)}, 109(2):I: 169–--212; II: 213–--258,
  1979.

\bibitem[JPSS81]{JPSS1981conducteur}
H.~Jacquet, I.~I. Piatetski-Shapiro, and J.~Shalika.
\newblock Conducteur des repr\'{e}sentations du groupe lin\'{e}aire.
\newblock {\em Math. Ann.}, 256(2):199--214, 1981.

\bibitem[Kha12]{Khan2012Simultaneous}
Rizwanur Khan.
\newblock Simultaneous non-vanishing of {$GL(3)\times GL(2)$} and {$GL(2)$}
  {$L$}-functions.
\newblock {\em Math. Proc. Cambridge Philos. Soc.}, 152(3):535--553, 2012.

\bibitem[Kim03]{Kim2003functoriality}
Henry~H. Kim.
\newblock Functoriality for the exterior square of {${\rm GL}_4$} and the
  symmetric fourth of {${\rm GL}_2$}.
\newblock {\em J. Amer. Math. Soc.}, 16(1):139--183, 2003.
\newblock With appendix 1 by Dinakar Ramakrishnan and appendix 2 by Kim and
  Peter Sarnak.

\bibitem[Lan13]{Lang2013algebraic}
Serge Lang.
\newblock {\em Algebraic number theory}, volume 110.
\newblock Springer Science \& Business Media, 2013.

\bibitem[LRS99]{LRS1999ramanujan}
Wenzhi Luo, Ze\'{e}v Rudnick, and Peter Sarnak.
\newblock On the generalized {R}amanujan conjecture for {${\rm GL}(n)$}.
\newblock In {\em Automorphic forms, automorphic representations, and
  arithmetic ({F}ort {W}orth, {TX}, 1996)}, volume~66 of {\em Proc. Sympos.
  Pure Math.}, pages 301--310. Amer. Math. Soc., Providence, RI, 1999.

\bibitem[Miy14]{Miyauchi2014Whittaker}
Michitaka Miyauchi.
\newblock Whittaker functions associated to newforms for {$GL(n)$} over
  {$p$}-adic fields.
\newblock {\em J. Math. Soc. Japan}, 66(1):17--24, 2014.

\bibitem[Mot93]{Mot1993explicit}
Yoichi Motohashi.
\newblock An explicit formula for the fourth power mean of the {R}iemann
  zeta-function.
\newblock {\em Acta Math.}, 170(2):181--220, 1993.

\bibitem[MV10]{MV2010subconvexity}
Philippe Michel and Akshay Venkatesh.
\newblock The subconvexity problem for {$\GL_2$}.
\newblock {\em Publications Math\'ematiques de l'IH\'ES}, 111:171--271, 2010.

\bibitem[Pal12]{Palm2012thesis}
Marc Palm.
\newblock Explicit {$\text{GL}(2)$} trace formulas and uniform, mixed {W}eyl
  laws.
\newblock {\em G\"ottingen: Univ. G\"ottingen (Diss.)}, 2012.

\bibitem[Sta01]{Stade2001Mellin}
Eric Stade.
\newblock Mellin transforms of {$GL (n,\Rr)$} {W}hittaker functions.
\newblock {\em American journal of mathematics}, 123(1):121--161, 2001.

\bibitem[Wu14]{Wu2014Burgess}
Han Wu.
\newblock Burgess-like subconvex bounds for {$\text{GL}_2\times\text{GL}_1$}.
\newblock {\em Geom. Funct. Anal.}, 24(3):968--1036, 2014.

\bibitem[Zac19]{Zach2019Periods}
Raphaël Zacharias.
\newblock {Periods and Reciprocity I}.
\newblock {\em International Mathematics Research Notices}, 06 2019.
\newblock rnz100.

\end{thebibliography}
\end{document}